\documentclass[11pt]{article}
\usepackage{graphicx}

\usepackage{amssymb,latexsym,amsmath,enumerate,verbatim,amsfonts,amsthm,cite,lscape,algorithm,algorithmic}
\usepackage{color} % added by TK
\usepackage[active]{srcltx}

\newtheorem{definition}{Definition}[section]

\newtheorem{theorem}{Theorem}[section]

\newtheorem{proposition}{Proposition}[section]
\newtheorem{remark}{Remark}[section]
\newtheorem{example}{Example}[section]

\def\R{{\mathbb{R}}}

\def\hat{\widehat}

\def\tilde{\widetilde}

\def\tto{\;{\lower 1pt \hbox{$\rightarrow$}}\kern -10pt
\hbox{\raise 2pt \hbox{$\rightarrow$}}\;}

\begin{document}

% REQUIRED

\title{\sf Convexifiability of  Continuous  and Discrete Nonnegative Quadratic Programs for Gap-Free Duality\thanks{ Research was partially supported by a grant from the Australian Research Council. }}

\date{\today}

\author{N. H. Chieu \thanks{Institute of Natural Sciences Education, Vinh University, Nghe An, Vietnam. Email: {chieunh@vinhuni.edu.vn}. Work of this author was carried out while he was at the University of New South Wales, Sydney, Australia.}
\and V. Jeyakumar\thanks{Department of Applied
Mathematics, University of New South Wales, Sydney 2052, Australia.
Email: {v.jeyakumar@unsw.edu.au}}
\and G. Li\thanks{Department of Applied
Mathematics, University of New South Wales, Sydney 2052, Australia. Email: {g.li@unsw.edu.au}. Date: September 19, 2018.}
 }

%
% Dianne Doe\thanks{Imagination Corp., Chicago, IL
% (\email{ddoe@imag.com}).}
% \and Paul T. Frank\thanks{Department of Applied Math, Fictional
% University, Boise, ID (\email{ptfrank@fictional.edu},
% \email{jesmith@fictional.edu}).}
% \and Jane E. Smith\footnotemark[3]}
%
%
% \author{, \  \ }
\maketitle
\begin{abstract}
%  In this paper, we examine exact copositive relaxations for nonconvex quadratic optimization problem with multiple quadratic constraints and nonnegative variables.
% Convexity underpins many important developments of mathematical theory and methods of optimization. When it comes to studying hard nonconvex optimization problems, one of the most desirable properties of these problems is undoubtedly convexifiability. The convexifiable programs often admit equivalent convex conic programming reformulations or exact convex progamming relaxations or gap-free Lagrangian-type dual problems under suitable conditions and consequently allow hierarchies of numerically tractable approximations for solving these problems.
% In this paper we examine nonnegative quadratically constrained quadratic programs and show that a key convexifiability property of these programs guarantees an exact copositive relaxation.
In this paper we show that a convexifiability property of nonconvex quadratic programs with nonnegative variables and quadratic constraints guarantees zero duality gap between the quadratic programs and their semi-Lagrangian duals. More importantly, we establish that this convexifiability is hidden in classes of nonnegative homogeneous quadratic programs and discrete quadratic programs, such as mixed integer quadratic programs,  revealing zero duality gaps. As an application, we prove that  robust counterparts of uncertain mixed integer quadratic programs with objective data uncertainty enjoy zero duality gaps under suitable conditions. Various sufficient conditions for convexifiability are also given.

% We then establish that, under mild assumptions, many classes of nonnegative quadratic programs admit the required convexifiability or exact copositive relaxations, including the classes of quadratic programs with mixed integer variables, robust counterpart of mixed integer quadratic programs in the face of objective data uncertainty, homogeneous programs with a single constraint,  and uniform quadratic programs with quadratic constraints.
%We also establish that uniform quadratic programs with nonnegative and quadratic constraints admit convexifiability under suitable conditions and  as an application we provide a copositive reformulation for the problem of finding the smallest radius ball  enclosing a given nonnegative intersection of balls. %We provide various examples illustrating our results.

\medskip

\medskip

{\bf AMS Classfication:}
 90C20, 90C26, 90C31
%\end{AMS}
\end{abstract}
\section{Introduction}

%The convexifiable programs often permit equivalent convex conic programming reformulations or exact convex programming relaxations under suitable conditions and consequently may allow hierarchies of numerically tractable approximations for solving these problems.

% Recently, equivalence of semi-Lagrangian dual of nonconvex quadratic programs with quadratic and linear constraints to natural copositive relaxation has been given  in Bomze~\cite{B15}  and completely positive relaxations for more general quadratic optimization problems with conic constraints have also been presented in \cite{BMP16}.

In this paper, we examine the quadratically constrained quadratic optimization problems (QPs) with nonnegative variables of the form
$$(P_1) \quad \begin{array}{rl} &\displaystyle\inf_{x\in \R^n} x^TAx+b^Tx+c\\
&\  \mbox{s.t.}\quad   x^TA_ix+b_i^Tx+c_i \leq 0,\,  i=0,1,...,m,\,  x_j \ge 0, j=1,2,\ldots,n,\end{array}$$
%\begin{equation*} \begin{array}{rl} &&(P) \quad \inf & x^TAx+b^Tx+c \\
%&&\mbox{s.t.} & x^TA_ix+b_i^Tx+c_i \leq 0,\,  i=0,1,...,m,\, \\
%&&&x_i\ge 0, i=1,2, \ldots, n,\end{array}
%\end{equation*}
where $A$,  $A_i$ are $(n\times n)$ symmetric matrices and $b,b_i\in \R^n,$ $c, c_i\in \R,$ $i=0,1,...,m$. The model problems of the form $(P)$ appear in broad areas of commerce, science and engineering where optimization is used. In particular, many classes of mixed integer quadratic programs and robust quadratic programs, such as the deterministic models of quadratic programs under data uncertainty, that arise frequently in real-world applications, can equivalently be reformulated as quadratic programs  of the form $(P_1)$.

The semi-Lagrangian dual of $(P_1)$ is given as (see Bomze~\cite{B15})
\begin{equation}\label{eq:D_1} (D_1) \quad \sup \limits_{u} \Theta(u)\quad \mbox{s.t.}\ \,  u\in \R^{m+1}_+,\end{equation}
where  $\Theta(u)$ is given by $\Theta(u):=\inf\limits_{x\in \R^n_+}L(x,u)$ with $L(x,u):=f(x)+\sum\limits_{i=0}^mu_ig_i(x)$ and  $f(x)=x^TAx+b^Tx+c,$  $g_i(x)=x^TA_ix+b_i^Tx+c_i,$ $i=0,1,...,m.$ It follows from the construction of $(D_1)$ that \begin{equation}\label{eq13bs}
\inf (P_1)\geq  \sup (D_1).\end{equation}

% Further, its dual  is a completely positive relaxation having  the dimension and the number of constraints less than the corresponding ones of the completely positive relaxation given in \cite{BMP16}.
%Especially, the copositive program  relaxation approaches proved to be powerful tools for examining and solving  classes of challenging quadratic optimization problems \cite{B12, B15,  D10}.
%We refer the readers to Bomze \cite{B12},  Burer \cite{Bu15} and D${\rm\ddot{u}}$r \cite{D10} for comprehensive surveys of copositive optimization and its applications.

The problem $(P_1)$ is said to admit {\bf zero duality gap}  whenever the optimal values of $(P_1)$ and its
semi-Lagrangian dual  problem $(D_1)$ are equal, i.e. $\inf (P_1) =  \sup (D_1)$. Unfortunately, zero duality gap between problems $(P_1)$ and $(D_1)$ does not always hold (see Example~\ref{e-inexact}). Some sufficient conditions for strong duality, i.e. $\inf (P_1)= \max (D_1)$, between the problems $(P_1)$ and $(D_1)$ have been given using
a generalized Karush-Kuhn-Tucker condition and copositivity of the related slack matrix \cite[Theorem 5.1]{B15}.  The semi-Lagrangian dual of nonconvex quadratic programs with quadratic and linear constraints is known to provide, in general, a better bound comparing to the standard Lagrangian dual and it
admits a natural copositive program reformulation \cite{B15}.
%A copositive programming problem is a conic linear program where a linear function is minimized over the  cone of copositive matrices   s.t.  linear constraints \cite{B12, B15, BDKR00}
%and its dual is known as a completely positive program.
Copositive programs have been extensively studied in the framework of relaxation schemes for solving  optimization problems in \cite{BMP16, BJL17,  Bu09, Bu12, Bu15, D10, QDRT}.

It is widely known that convexity of sets and functions of optimization problems underpins many important developments of mathematical theory and methods of optimization. For instance, recent research (see \cite{BJL17,jl14,jl16}) has examined the role of convexity in duality and exact conic programming relaxations for special classes of quadratic programs such as extended trust-region problems, CDT problems (two-balls trust-region problems) and separable minimax quadratic programs. When it comes to studying duality for hard nonconvex quadratic programs, such as  general nonnegative quadratic programs with quadratic constraints and mixed integer quadratic programs, identifying the key features that underline the zero duality gap property and then finding classes of quadratic programs that possess the features and zero duality gaps are undoubtedly important.

In this paper we show that a convexifiability property (see Definition 2.1) of general nonconvex quadratic programs with nonnegative variables guarantees zero duality gap between
quadratic programs with nonnegative variables and its semi-Lagrangian dual. More importantly, we establish that this convexifiability is hidden in classes of nonnegative homogeneous quadratic programs and discrete quadratic programs, such as mixed integer quadratic programs,  revealing zero duality gaps. More specifically, our main contributions include the following:
\begin{itemize}
\item[(i)] By introducing the idea of convexifiability for $(P_1)$, we first establish that zero duality gap holds between $(P_1)$ and its semi-Lagrangian dual $(D_1)$, whenever the problem $(P_1)$ is convexifiable. In particular, we show that a nonconvex homogeneous quadratic optimization problem with a single strictly copositive quadratic constraint and nonnegative variables enjoys convexifiability and consequently zero duality gap.
%In the case where the problem $(P_1)$ attains its infimum, we show that our geometric condition guarantees related known conditions for exact copositive relaxation.

\item[(ii)] We also prove that quadratic programming problems with mixed integer variables admits hidden convexifiablity in the sense that its equivalent continuous quadratic program reformulation is convexfiable,
under mild assumptions.
Consequently, we obtain zero duality gap between a mixed integer quadratic program and its semi-Lagrangian dual, recovering and complementing the important copositive representation result given recently in \cite{Bu09}.

\item[(iii)] As an important application, we then establish that zero duality gap holds for a class of
%
%single as well as multi-stage
robust mixed integer quadratic programs with objective data uncertainty under suitable conditions. Robust optimization approach, which treats continuous  optimization problems with parameters of unknown but fixed value, is now relatively well understood (see \cite{robust_book, Bersimas,cjl,goberna,jeya-li-robust1}). Extending the robust optimization techniques to an optimization problem with mixed integer constraints is increasingly becoming a cutting-edge research area in optimization under data uncertainty  %It applies to
    %dynamic decision-making problems, where the decision makers are able to adjust their strategies to information revealed over time
    (see \cite{robust_book,Bersimas,Teo} and other references therein).
   % Our results show that existing conic optimization techniques may be utilized to analyze and approximate the challenging robust mixed-integer quadratic optimization problems. These results also potentially lead to  hierarchies of tractable relaxations \cite{B15,BJL17} for solving robust mixed-integer quadratic optimization problems.

% \item[(iv)] We  finally provide conditions under which nonnegative uniform quadratic  optimization problems, where the Hessian of a quadratic constraint function is a scalar multiple of the Hessian of the objective function, enjoys convexifiability.
%  While exact semi-definite relaxations for uniform quadratic optimization problems have recently been investigated in the literature \cite{jl14,jl16,B07}, the presence of non-negativity constraint makes the uniform quadratic optimization problem, in general, NP-hard and violates the standard requirements, such as the dimension conditions of \cite{jl14,jl16},  for exact conic relaxations.
  %As an application, we present a copositive reformulation for the problem of finding the smallest radius ball enclosing a given nonnegative intersection of balls.

%   \item [(iii)] We also establish that exact copositive relaxation holds for nonnegative extended trust-region problems under easily verifiable conditions, satisfying the geometric property.
  \end{itemize}

The outline of the paper is as follows. Section 2 presents zero duality gap results between a  nonnegative quadratic program and its semi-Lagrangian dual, under convexifiability.
%Section 3 establishes min-max semi-Lagrangian duality and copositive relaxation, where the optimal value, $\inf(P_1)$, is attained.
Section 3 examines hidden convexifiability and zero duality gaps for quadratic programs with mixed integer variables.
Section 4 provides an important application of our duality to robust optimization, where we  establish that zero duality gaps hold for robust mixed-integer quadratic optimization problems under objective data uncertainty.
Section 5 gives further technical conditions for convexifiability of nonnegative quadratic programs and the historical links between our duality and exact copositive relaxations of quadratic programs.
Finally, Section 6 makes concluding statements with comments on future work.
%Appendix provides the links between our exact copositive relaxation results and the semi-Lagrangian
%duality results of the recent literature under the convex-like geometry condition.

\section{Convexifiability of Nonnegative QPs and Duality}

%In this section,  we  examine  exact copositive relaxations and  semi-Lagrangian duality,  under a geometrical condition,  for a broad class of   nonconvex  quadratic programs with nonnegativity constraints.

Consider the following nonconvex quadratic optimization problem:

$$  (P_1)\quad \begin{array}{rl} &\displaystyle\inf_{x} x^TAx+b^Tx+c\\
&\  \mbox{s.t.}\quad  x\in \R^n_+,\  x^TA_ix+b_i^Tx+c_i \leq 0,\,  i=0,1,...,m,\end{array}$$
%\begin{equation*} \begin{array}{rl} &&(P) \quad \inf & x^TAx+b^Tx+c \\
%&&\mbox{s.t.} & x^TA_ix+b_i^Tx+c_i \leq 0,\,  i=0,1,...,m,\, \\
%&&&x_i\ge 0, i=1,2, \ldots, n,\end{array}
%\end{equation*}
where $A, A_i\in S^n$ are $(n\times n)$ symmetric matrices and $b,b_i\in \R^n,$ $c, c_i\in \R,$ $i=0,1,...,m$. Here $S^n$ denotes the set of  $(n\times n)$ symmetric matrices. {\it In the sequel we always assume that the feasible set of the problem $(P_1)$ is non-empty.}

For the problem $(P_1)$, let  $f(x)=x^TAx+b^Tx+c,$  $g_i(x)=x^TA_ix+b_i^Tx+c_i,$ $i=0,1,...,m$, and
$$\mathcal{A}_{P_1}:=\big\{\big(g_0(x), g_1(x),..., g_m(x), f(x)\big) : x\in\R^n_+\big\}+\R^{m+2}_+\ . $$
% and
% $$\mathcal{A}_{P_1}^{\epsilon}:=\big\{\big(g_0(x), g_1(x),..., g_m(x), f(x)-\inf(P_1)+\epsilon\big) : x\in\R^n_+\big\}+\R^{m+2}_+$$
%By construction, clearly, for each $\epsilon>0$,
%$(0_{{\mathbb{R}^{m+1}}}, \inf(P_1)-\epsilon) \notin \mathcal{A}_{P_1}$ (otherwise, there exists $x\in \R^n_+$ such that $g_i(x) \le 0$, $i=0,1,\ldots,m$ and $f(x) \le \inf(P_1)-\epsilon$ which
%is impossible).
We now define a key geometric property, called convexifiablity, which will play a key role in establishing zero duality gap between $(P_1)$ and $(D_1)$. Recall that for a set $A$,
its closure and convex hull are denoted by ${\rm cl} A$ and ${\rm conv} A$ respectively.
\begin{definition}[{\bf Convexifiability}] The problem $(P_1)$ is said to be {\bf convexifiable} whenever
\[
\big(\{0_{{\mathbb{R}^{m+1}}}\} \times \mathbb{R}\big) \cap {\rm cl \, conv} \mathcal{A}_{P_1}= \big(\{0_{{\mathbb{R}^{m+1}}}\} \times \mathbb{R}\big) \cap  \mathcal{A}_{P_1}.
\]
The problem $(P_1)$ is said to be {\bf strongly convexifiable} whenever the set $\mathcal{A}_{P_1}$ is closed and convex.

%
%for all $\epsilon>0$,  \[
%(0_{{\mathbb{R}^{m+1}}}, \inf(P_1)-\epsilon) \notin   {\rm cl \, conv} \mathcal{A}_{P_1}.
%\]
\end{definition}
It easily follows from the definition that if the problem $(P_1)$ is strongly convexifiable then it is convexifiable.  We will see in Sections 3-4, convexifiability can be satisfied by many important classes of  specially structured optimization problems, such as mixed integer quadratic optimization problems, robust counterparts of uncertain mixed integer quadratic optimization problems in the face of objective data uncertainty. However, the strong convexifiability is often much harder to be satisfied than convexifiability even for problems in one or two dimensions as we see in the following simple examples.

% As we saw in Lemma~\ref{lm31}, the closed convexity requirement of the set $\mathcal{A}_{P_1}$ is satisfied by nonnegative homogeneous quadratic programs with a single quadratic constraint under suitable conditions and it is a sufficient condition in general for convexifiability of quadratic programs $(P_1)$. However, this requirement is much stronger than the convexifiability property as we see in the following simple examples.

%We have seen that the closedness and convexity of
%$\mathcal{A}_{P_1}$ is satisfied for nonconvex homogeneous quadratic program with single quadratic constraints and nonnegative variables, under a strict copositive condition. This, in
%particular, implies the convexifiability of this problem. On the other hand, as we will see from the following two simple examples, convexifibility is, in general,  a strictly weaker requirement than the closedness and convexity of
%$\mathcal{A}_{P_1}$, even for problems of one or two dimension.
\begin{example}\label{count1}
Consider the problem $(P_1)$, where $g_0(x)=x(x-1)$, $g_1(x)=-x(x-1)$, $f(x)=x^2$. Then, $\mathcal{A}_{P_1}=\{(g_0(x),g_1(x),f(x)): x \ge 0\}+\R^3_+$.
We first see that  problem $(P_1)$ is convexifiable. Noting that $\mathcal{A}_{P_1} \subseteq \mathbb{R}^2 \times [0,+\infty)$, one has $\big(\{0_{{\mathbb{R}^{2}}}\} \times \mathbb{R}\big) \cap {\rm cl \, conv} \mathcal{A}_{P_1} \subseteq \{0_{\R^2}\} \times [0,+\infty)$. Moreover, it can be directly verified that $\big(\{0_{\R^2}\} \times \mathbb{R}\big) \cap  \mathcal{A}_{P_1}=\{0_{\R^2}\} \times [0,+\infty)$.
So, we must have $\big(\{0_{{\mathbb{R}^{2}}}\} \times \mathbb{R}\big) \cap {\rm cl \, conv} \mathcal{A}_{P_1}= \big(\{0_{{\mathbb{R}^{2}}}\} \times \mathbb{R}\big) \cap  \mathcal{A}_{P_1}=\{0_{\R^2}\} \times [0,+\infty)$. Hence the problem is convexifiable.
 %and hence, the underlying problem is convexifiable.

On the other hand, direct verification shows that $(0,0,0) \in \mathcal{A}_{P_1}$ and $(2,-2,4) \in \mathcal{A}_{P_1}$ (consider $x=0$ and $x=2$ respectively). But
their mid point $(1,-1,2) \notin \mathcal{A}_{P_1}$ (otherwise, there exists $x\ge 0$ such that
\[
x(x-1) \le 1, -x(x-1) \le -1 \mbox{ and } x^2 \le 2
\]
The first two relations imply that $x(x-1)=1$ and so, $x=\frac{\sqrt{5}+1}{2}$ (as $x \ge 0$). This contradicts the fact that $x^2 \le 2$.
 So, $\mathcal{A}_{P_1}$ is not convex and hence the problem is not strongly convexifiable.
\end{example}
\begin{example}\label{count2}
Consider the problem $(P_1)$, where $g_0(x)=-x_1x_2$, $f(x)=x_1$. Then, $\mathcal{A}_{P_1}=\{(g_0(x),f(x)): x \in \R^2_+\}+\R^2_+$.Noting that $\mathcal{A}_{P_1} \subseteq \mathbb{R} \times [0,+\infty)$, one has $\big(\{0\} \times \mathbb{R}\big) \cap {\rm cl \, conv} \mathcal{A}_{P_1} \subseteq \{0\} \times [0,+\infty)$. Moreover, it can be directly verified that $\big(\{0\} \times \mathbb{R}\big) \cap  \mathcal{A}_{P_1}=\{0\} \times [0,+\infty)$. So, we must have $\big(\{0\} \times \mathbb{R}\big) \cap {\rm cl \, conv} \mathcal{A}_{P_1}= \big(\{0\} \times \mathbb{R}\big) \cap  \mathcal{A}_{P_1}$, and hence, the problem $(P_1)$ is convexifiable.

On the other hand, direct verification shows that $(-1,0) \notin \mathcal{A}_{P_1}$; while $(-1,\frac{1}{k}) \in \mathcal{A}_{P_1}$ for all $k \in \mathbb{N}$ (by considering $x_1=\frac{1}{k}$ and $x_2=k$). So,
$\mathcal{A}_{P_1}$ is not closed and the problem is not strongly convexifiable.
\end{example}

%We first  establish an exact copositive relaxation result for $(P_1)$ under convexifiability and present a convex-like geometric condition guaranteeing convexifiability.
%We will then justify at the end of this section that the homogeneous nonnegative quadratic programming with single constraints satisfied the

Recall that the semi-Lagrangian dual of $(P_1)$ is given by
$$(D_1) \quad \sup \limits_{u} \Theta(u)\quad \mbox{s.t.}\ \,  u\in \R^{m+1}_+,$$
where  $\Theta(u)$ is defined as  $\Theta(u):=\inf\limits_{x\in \R^n_+}L(x,u)$ with $L(x,u):=f(x)+\sum\limits_{i=0}^mu_ig_i(x)$ and  $f(x)=x^TAx+b^Tx+c,$  $g_i(x)=x^TA_ix+b_i^Tx+c_i,$ $i=0,1,...,m.$
So, by construction, we see that  \begin{equation}\label{eqbz}\inf (P_1)\geq \sup(D_1).\end{equation}

We now show that there is no duality gap between $(P_1)$ and $(D_1)$ whenever the problem $(P_1)$ is convexifiable.

\begin{theorem}\label{lemma2.1} {\bf (Zero Duality Gaps via Convexifiability)} If the problem~$(P_1)$ is convexifiable then, we have
 $$ \inf (P_1)= \sup (D_1).$$
 In particular, if the problem~$(P_1)$ is strongly convexifiable  and $\inf(P_1)>-\infty$, then
 $$ \min (P_1)=\sup (D_1).$$
\end{theorem}
\begin{proof}
If $\inf(P_1)=-\infty$ then the conclusion immediately follows from \eqref{eqbz}. As the problem $(P_1)$ is feasible,  $\inf(P_1)$ is finite.
Fix any $\epsilon>0$. By construction,
$(0_{{\mathbb{R}^{m+1}}}, \inf(P_1)-\epsilon) \notin \mathcal{A}_{P_1}$ (otherwise, there exists $x\in \R^n_+$ such that $g_i(x) \le 0$, $i=0,1,\ldots,m$ and $f(x) \le \inf(P_1)-\epsilon$ which
is impossible). As problem $(P_1)$ is convexifiable, we see that
\begin{eqnarray}\label{eq:conclu}
(0_{{\mathbb{R}^{m+1}}}, \inf(P_1)-\epsilon) &\notin & {\rm cl \, conv} \mathcal{A}_{P_1}. % \\
 %& = & {\rm cl \, conv} \bigg(\big\{\big(g_0(x), g_1(x),..., g_m(x), f(x)\big) : x\in\R^n_+\big\}+\R^{m+2}_+\bigg),
\end{eqnarray}
Indeed, if this not the case, that is, $(0_{{\mathbb{R}^{m+1}}}, \inf(P_1)-\epsilon) \in {\rm cl \, conv} \mathcal{A}_{P_1}$, then we see that
$(0_{{\mathbb{R}^{m+1}}}, \inf(P_1)-\epsilon) \in \big(\{0_{{\mathbb{R}^{m+1}}}\} \times \mathbb{R}\big) \cap {\rm cl \, conv} \mathcal{A}_{P_1}$. Then, the convexifiability assumption gives
us that $(0_{{\mathbb{R}^{m+1}}}, \inf(P_1)-\epsilon) \in \big(\{0_{{\mathbb{R}^{m+1}}}\} \times \mathbb{R}\big) \cap \mathcal{A}_{P_1}$ which makes contradiction. So, \eqref{eq:conclu} holds.
The strong convex separation theorem implies that there exist $\mu_i \in \mathbb{R}$, $i=0,1,\ldots,m$ and $\mu \in \mathbb{R}$ with $(\mu_0,\mu_1,\ldots,\mu_{m},\mu) \neq 0_{\mathbb{R}^{m+2}}$ such that
\[
\mu (\inf(P_1)-\epsilon) < \sum_{i=0}^m \mu_i u_i + \mu r \mbox{ for all } (u_0,u_1,\ldots,u_m,r) \in \mathcal{A}_{P_1}.
\]
This implies that $(\mu_0,\mu_1,\ldots,\mu_{m},\mu) \in \mathbb{R}^{m+2}_+ \backslash \{0_{\mathbb{R}^{m+2}}\}$ and
\[
\mu \big(\inf(P_1)-\epsilon\big) < \sum_{i=0}^m \mu_i g_i(x) + \mu f(x) \mbox{ for all } x \in \mathbb{R}^n_+.
\]
We observe that $\mu>0$. (Otherwise, one has $(\mu_0,\ldots,\mu_m) \neq 0_{\mathbb{R}^{m+1}}$ and
\[
\sum_{i=0}^m \mu_i g_i(x) >0 \mbox{ for all } x \in \mathbb{R}^n_+.
\]
Let $x_0$ be a feasible point of $(P_1)$. Then, $g_i(x_0) \le 0$ and $\sum_{i=0}^m \mu_i g_i(x)\le 0$ which is impossible.) Thus, by dividing $\mu$ on both sides, one has

\[
f(x)+\sum_{i=0}^m \bar u_i g_i(x)-(\inf(P_1)-\epsilon)>0 \mbox{ for all } x \in \mathbb{R}^n_+,
\]
where $\bar u_i= \frac{\mu_i}{\mu} \ge 0$, $i=0,1,\ldots,m$. Consider
\begin{equation} (D_1) \quad \sup \limits_{u} \Theta(u)\quad \mbox{s.t.}\ \,  u\in \R^{m+1}_+,\end{equation}
where  $\Theta(u)$ is given by $\Theta(u):=\inf\limits_{x\in \R^n_+}L(x,u)$ with $L(x,u):=f(x)+\sum\limits_{i=0}^mu_ig_i(x)$ and  $f(x)=x^TAx+b^Tx+c,$  $g_i(x)=x^TA_ix+b_i^Tx+c_i,$ $i=0,1,...,m.$
% From the construction, it can be verified that \begin{equation}
%\sup (D_1)=\sup_{(y_0,u) \in \mathbb{R} \times \mathbb{R}_+^{m+1}}\{y_0: \inf\limits_{x\in \R^n_+}L(x,u) \ge y_0\}=\sup (CP_1^*) .\end{equation}
This implies that, for each $\epsilon>0$,
\[
\sup(D_1)=\sup_{u \in \mathbb{R}^{m+1}_+}\inf_{x \in \mathbb{R}^n_+}L(x,u)\geq \inf_{x \in \mathbb{R}^n_+}L(x,\bar u) \ge \inf(P_1)-\epsilon.
\]
Letting $\epsilon \rightarrow 0$, one has $\sup(D_1) \ge \inf(P_1)$. Therefore, the conclusion follows from \eqref{eqbz}.

Now assume that the problem $(P_1)$ is strongly convexifiable (i.e. $\mathcal{A}_{P_1}$ is closed and convex) and $\inf (P_1)>-\infty$. Then, $$\big(\{0_{{\mathbb{R}^{m+1}}}\} \times \mathbb{R}\big) \cap {\rm cl \, conv} \mathcal{A}_{P_1}= \big(\{0_{{\mathbb{R}^{m+1}}}\} \times \mathbb{R}\big) \cap  \mathcal{A}_{P_1},$$ that is,
%from the definition of $\inf(P_1)$, we see that, for each $\epsilon>0$, $(0_{\mathbb{R}^{m+1}},\inf(P_1)-\epsilon) \notin \mathcal{A}_{P_1}$.   So, by the assumption that
%$\mathcal{A}_{P_1}$ is closed and convex, we see that  $(0_{\mathbb{R}^{m+1}},\inf(P_1)-\epsilon) \notin \mathcal{A}_{P_1}= {\rm cl \, conv}\mathcal{A}_{P_1}$, and thus,
problem~$(P_1)$ is convexifiable. This guarantees that $ \inf (P_1)= \sup (D_1).$
To finish the proof, it remains to show that the minimum in $(P_1)$ is  attained. To see this, let $x^{(k)}$ be feasible for $(P_1)$ such that
$f(x^{(k)}) \rightarrow \inf(P_1)$.  Then, $g_i(x^{(k)}) \le 0$, $i=0,1,\ldots,m$ and so, $(0_{\mathbb{R}^{m+1}},f(x^{(k)})) \in \mathcal{A}_{P_1}$. As $\mathcal{A}_{P_1}$ is closed,
we see that its limit $(0_{\mathbb{R}^{m+1}},\inf (P_1)) \in \mathcal{A}_{P_1}$. This shows that there exists $\bar x \in \mathbb{R}^n_+$ such that $\bar x$ is feasible for $(P_1)$ and $f(\bar x)=\inf(P_1)$. In other
 words, $\bar x$ is a solution for $(P_1)$. % Moreover, let $\bar X=\left(\begin{array}{c}
%                                                                         1 \\
%                                                                         \bar x
%                                                                        \end{array}
%  \right)\left(\begin{array}{c}
%                                                                         1 \\
%                                                                         \bar x
%                                                                        \end{array}
%  \right)^T$. Then, ${\rm Tr}(H\bar X)=f(\bar x)=\inf(P_1)=\inf(CP_1)$. This implies that the minimum of $(CP_1)$ is also attained.
\end{proof}

% We now see. as an immediate consequence, that the convexity condition, ``$\mathcal{A}_{P_1}$ is closed and convex'' implies the required convexfiability, and in particular,
% an exact copositive relaxation.

In passing it is worth noting that various forms of convex-like conditions of non-convex programs have been utilized to obtain exact semi-definite and exact second-order cone relaxation results recently
 for specially structured nonconvex quadratic optimization problems including the extended trust region problems \cite{jl14,jl16} and quadratic problems with a single constraint \cite{Fabian1,Fibian2}.

We now present a simple one-dimensional example illustrating an infinite duality gap for a  nonnegative quadratic program that is not convexifiable.

\begin{example}\label{e-inexact}{\rm {\bf (Failure of zero duality gaps without convexifiability)}
Consider the one-dimensional nonconvex quadratic optimization problem
$$(E_1)\quad \begin{array}{rl}& \inf\  -x^2 \\
& \mbox{s.t.} \  x-1 \le 0, \\
& \quad\ \ x \ge 0.
\end{array}
$$
Clearly, the optimal value of $(E_1)$ is $-1$ and the optimal solution is $x=1$.

Its semi-Lagrangian dual is
\[
(DE_1)\ \ \ \sup_{u \ge 0} \Theta(u)
\]
where $\Theta(u)=\inf_{x \in \mathbb{R}_+}\{-x^2+ u(x-1)\}$. It is not hard to see that for any $u \ge 0$, $\Theta(u)=-\infty$, and so, $\sup(DE_1)=-\infty$. Thus, there is an
infinite gap between the optimal values of $(E_1)$ and its semi-Lagrangian dual $(DE_1)$.

% Its complete positivity relaxation can be formulated as
% $$(CE_1) \quad\begin{array}{rl} & \inf\limits_{X \in \mathcal{C}} \  -X_{22} \\
% & \mbox{s.t.} \  X_{12} \le 1, \\
% & \, \  \quad X_{11}=1.\end{array}
% $$
% We first see that the semi-Lagrangian duality fails. Note that exact semi-Lagrangian duality implies the optimal value of $(E_1)$ equals to the optimal value
% of $(CE_1)$. To see semi-Lagrangian duality fails, we next see that there is an infinite gap between the optimal value of $(E_1)$ and the optimal value
% of $(CE_1)$. To this end, let
%  $X_k=\frac{1}{k}\left(\begin{array}{c}
% 1\\
% k
%               \end{array}
% \right) \left(\begin{array}{c}
% 1\\
% k
%               \end{array}
% \right)^T+(1-\frac{1}{k})\left(\begin{array}{c}
% 1\\
% 0
%               \end{array}
% \right) \left(\begin{array}{c}
% 1\\
% 0
%               \end{array}
% \right)^T
%  =\left(\begin{array}{cc}
%        1 & 1 \\
%        1 & k
%       \end{array}\right)$, $k \in \mathbb{N}$, which is feasible for $(CE_1)$. This shows that the optimal value of $(CE_1)$ is $-\infty.$
%       and so, the copositive relaxation has an infinite gap. As a result, semi-Lagrangian duality fails.

Direct verification shows that
\[
\mathcal{A}_{E_1}:=\big\{\big(x-1, -x^2\big) : x\ge 0\big\}+\R^{2}_+ = \big\{\big(x, -x^2\big) : x\ge 0\big\}+\{(-1,0)\}+\R^{2}_+,
\]
% nor
% \[
% \mathcal{A}_{E_2}^+:=\big\{\big(x-1, -x^2\big) : x\ge 0\big\}+{\rm int}\R^{2}_+ = \big\{\big(x, -x^2\big) : x\ge 0\big\}+\{(-1,0)\}+{\rm int}\R^{2}_+,
% \]
is not a convex set. Moreover, note that ${\rm cl}\, {\rm conv} \, \mathcal{A}_{E_1}=\{(x_1,x_2): x_1 \ge -1\}$. It follows that
\[
(\{0\}\times \mathbb{R}) \cap {\rm cl}\, {\rm conv} \, \mathcal{A}_{E_1}=\{(x_1,x_2): x_1 \ge -1\}= (\{0\}\times \mathbb{R})
\]
and
\[
(\{0\}\times \mathbb{R}) \cap  \mathcal{A}_{E_1}= \{0\}\times [-1,+\infty).
\]
%
% $\inf(E_1)=-1$ and for all $\epsilon>0$
%\[
%(0,-1-\epsilon) \in {\rm cl}\, {\rm conv} \, \mathcal{A}_{E_1}=\{(x_1,x_2): x_1 \ge -1\}.
%\]
Thus, we see that convexifiability condition fails for the problem $(E_1)$.}
\end{example}

\subsection*{Strong convexifiability of Homogeneous quadratic programs}
 Let us consider the nonconvex homogeneous quadratic program with single quadratic constraints and nonnegative variables (HQP)
 $$(HQP) \ \ \quad \min\{x^TAx: x \in \mathbb{R}^n_+, \, x^TBx \le 1 \},$$
where $A,B$ are  $(n\times n)$ symmetric matrices. Let $e \in \mathbb{R}^n$ be the vector whose elements are all one.
% and $B$ is a strict copositive matrix.
We now show that our geometric condition always  holds for the  nonconvex homogeneous quadratic program (HQP)
if $B$ is a strictly copositive matrix.

% The following lemma will play an essential role in proving the main result of this section.
\begin{proposition}\label{lm31} {\bf (Strong convexifiability of (HQP))} Let $A, B\in  S^n$ such that  $B$ is strictly copositive. Then, (HQP) is strongly convexifiable.

\end{proposition}
\noindent{\it Proof.}  %If $\K=\{0\},$ the conclusion is trivial. Suppose now that $\K\not=\{0\}.$
We first verify the closedness of the set
$$\Upsilon:=\big\{(x^TBx,x^TAx)\ |\ x\in \mathbb{R}^n_+ \big\}+\R^2_+.$$
To see this, let $(b_k,a_k) \in \Upsilon$ with $(b_k,a_k) \rightarrow (\bar b,\bar a)$. Then, there exists $\{x_k\} \subseteq \R^n_+$
such that $x_k^TBx_k \le b_k$ and $x_k^TAx_k \le a_k$. As $B$ is strictly copositive, $x_k \in \R^n_+$ and $x_k^TBx_k \le b_k \rightarrow \bar b$, we see that $\{x_k\}$ is bounded (Otherwise, by passing to subsequence
if necessary, we can assume that $\|x_k\| \rightarrow \infty$ and $\frac{x_k}{\|x_k\|} \rightarrow d$ with $d \in \R^n_+ \backslash\{ 0\}$. Then,
\[d^TBd=\lim_{k \rightarrow \infty}(\frac{x_k}{\|x_k\|})^TB\frac{x_k}{\|x_k\|} \le \lim_{k \rightarrow \infty} \frac{b_k}{\|x_k\|^2}=0.\]
This is impossible due to the strict copositive assumption of $B$.) By passing to subsequence, we assume that $x_k \rightarrow \bar x \in \R^n_+$. Then, letting $k \rightarrow \infty$, we see that
\[
(\bar b, \bar a) \in (\bar x^TB\bar x,\bar x^TA\bar x)+\R^2_+  \subset \Upsilon.
\]
So, $\Upsilon$ is closed.

We next show that $\Upsilon$ is convex. Denote $\alpha^*:=\min\{x^TAx: x \in \mathbb{R}^n_+,  x^TBx=1\}$.  Let us consider the following two  cases.

{\it Case 1:}  $\alpha^* \ge 0$. Then, by the strict copositivity of $B,$ we see that  $x^TAx \ge 0$ for all $x \in \mathbb{R}^n_+$ (otherwise, there exists $d \in \mathbb{R}^n \backslash \{0\}$
such that $d^TAd<0$. As $B$ is strict copositive, $d^TBd>0$. Let $\bar x=\frac{d}{\sqrt{d^TBd}}$. Then, $\bar x^TB\bar x=1$ and $\bar x^TA \bar x<0$. This shows that $\alpha^*<0$ which makes
contradiction.)
Thus, we have
$$\{\big(x^TBx,x^TAx\big):x \in \mathbb{R}^n_+\} \subseteq \mathbb{R}^2_+$$
 and so, $\Upsilon \subseteq  \R^{2}_+$. Note that $(0,0) \in \{\big(x^TBx,x^TAx\big):x \in \mathbb{R}^n_+\}$, and so,
$\Upsilon  \supseteq  \R^{2}_+$. This shows that $\Upsilon= \R^{2}_+$ which is convex.

{\it Case 2:} $\alpha^*<0$. We verify the convexity of $\Upsilon$ by showing that
$$\Upsilon=\{(t,\alpha^* t): t \ge 0\}+ \R^{2}_+.$$ To see this, let $(u,v) \in \Upsilon$. Then, there exists $x \in \mathbb{R}^n_+$
such that $x^TBx \le u$ and $x^TAx  \le v$.
Let
$t=x^TBx \ge 0$. From  the definition of $\alpha^*$ it follows  that
$x^TAx \ge \alpha^* t$. So, $t \le u$ and $\alpha^* t \le v$, that is, $(u,v) \in \{(t,\alpha^* t): t \ge 0\}+ \R^{2}_+$. Thus, $\Upsilon \subseteq \{(t,\alpha^* t): t \ge 0\}+ \R^{2}_+$.
On the other hand, let $(u,v) \in \{(t,\alpha^* t): t \ge 0\}+ \R^{2}_+$. Then,  there exists $t \ge 0$  such that $t<u$ and $ \alpha^* t \le v$.
%If $t=0$, then $u<0$ and $v<0$, and so, $(u,v) \in K$. So, we can assume $t>0$.
Let $x^* \in \mathbb{R}^n_+$ be a solution of $\min\{x^TAx: x \in \mathbb{R}^n_+, x^TBx=1\}$
and let $z=(z_1,\ldots,z_n)$ with $z_i= \sqrt{t} \, x_i^*$. Then, $z \in \mathbb{R}^n_+$, $z^TBz=t$ and $z^TAz= t (x^*)^TAx^*=\alpha^* t$. So,
$z^TBz \le u$ and $z^TAz \le v$, and hence $(u,v) \in \Upsilon$. Thus, the reverse inclusion also holds.
Consequently,  $\Upsilon$ is convex.  Thus, the $\Upsilon$ is closed and convex and so, the required convex-like geometric condition for convexifiability is satisfied by the problem (HQP).  $\hfill\Box$

\begin{remark}
In the special case of  $B=ee^T$, by introducing a nonnegative
slack variable,  (HQP) can be equivalently rewritten as the so-called
standard quadratic optimization problem, which is a well-known class of optimization problems admitting an exact copositive relaxation \cite{BDKR00}. The links between semi-Lagrangian duality and exact copositive relaxation are given in Appendix later in the paper.
\end{remark}

The conclusion of Proposition~\ref{lm31} may fail if the strict copositivity assumption is removed.
{\rm  \begin{example}{\bf (Failure of strong convexifiability without strict copositivity)} {\rm Let
$B=\begin{pmatrix} 1& 1\\ 1 &-1\end{pmatrix}$ and  $A=\begin{pmatrix} -2& 1\\ 1 &1\end{pmatrix}$.
We first observe that $A, B$ are symmetric matrices and neither $A$ nor $B$ is copositive because they both have negative diagonal elements.
We now see that
$$\begin{array}{rl}\Upsilon&:=\big\{(x^TBx,x^TAx)\ |\ x\in \R^2_+\big\}+ \R^2_+\\
&=\big\{(x_1^2-x_2^2+2x_1x_2,-2x_1^2+x_2^2+2x_1x_2)\ |\ (x_1,x_2)\in \R^2_+\big\}+ \R^2_+\end{array}$$  is nonconvex. To see this, let
$g(x)=x_1^2-x_2^2+2x_1x_2$ and $f(x)=-2x_1^2+x_2^2+2x_1x_2$. Note that $g(0,1)=-1$, $f(0,1)=1$, $g(1,0)=1$ and $f(1,0)=-2$. We see that
\[
(-1,1) \in \Upsilon \mbox{ and } (1,-2) \in \Upsilon.
\]
We now verify that their midpoint $(0,-1/2)=\frac{(-1,1)+(1,-2)}{2} \notin \Upsilon$. Suppose to the contrary that
$(0,-1/2) \in \Upsilon$. Then, there exists $(x_1,x_2)$ such that
% we note that, for any $\epsilon>0$,
% $(n^2+1,-2n^2+1) \in \Upsilon$ and $(-n^2+1,n^2+1) \in \Upsilon$. But their mid point $(1,1-\frac{n^2}{2})$.
%suppose to the contrary that $\Upsilon$ is convex. Note that the system
$$\begin{cases} x_1^2-x_2^2+2x_1x_2 \le 0,\\
-2x_1^2+x_2^2+2x_1x_2 \le -1/2, \\
x_1\geq 0,\  x_2\geq 0.
\end{cases}
$$
%We first note that $x_2 \neq 0$ (otherwise, the first and second inequalities imply that $x_1=0$ and $2x_1^2 \ge 1/2$ which is impossible).
The first inequality gives us that
$(x_1+x_2)^2 \le 2x_2^2$ and so, $x_1 \le (\sqrt{2}-1)x_2$. On the other hand, adding the first and second inequality, one has
\[
x_1(-x_1+4x_2)=-x_1^2+4x_1x_2 \le -1/2<0,
\]
which implies that $x_1>4 x_2.$ So, we have $4x_2<x_1 \le (\sqrt{2}-1)x_2$ which cannot happen due to $x_1,x_2 \ge 0$. This contradiction shows that $\Upsilon$ is not convex.
% has no solution. So, $\inf\{x_1^2-x_2^2+2x_1x_2: -2x_1^2+x_2^2+2x_1x_2 \le 0, x_1,x_2 \ge 0\} \ge 0$. Note that Slater condition is satisfied with $\hat{x}=(2,1)$ for this
% optimization problem. Thus, by Theorem~\ref{thm2.1}, the optimal value of its copositive relaxation is attained and is larger or equal to $0$. So, there exists some $t \ge 0$ such that
% $$\begin{pmatrix} 0&0&0\\
% 0&1& 1\\
% 0&1 &-1\end{pmatrix}+t\begin{pmatrix} 0&0&0\\
% 0&-2& 1\\
% 0&1 &1\end{pmatrix}\ \mbox{is  $\R^3_+$-copositive}.$$
% This implies that
% $$ \begin{pmatrix}
% 1-2t& 1+t\\
% 1+t &-1+t\end{pmatrix}\ \mbox{is $\R^2_+$-copositive}.$$
% In particular, we would get $\frac{1}{2}\geq t\geq 1.$ This contradiction shows that $\Upsilon$ is nonconvex.
}\end{example}

We will now show in the following sections that convexifiability property can be satisfied for several important and challenging quadratic programs under mild assumptions including the quadratic programs with mixed integer variables
and robust mixed integer quadratic programs under objective data uncertainty.
\section{ Hidden Convexifiability of Discrete Quadratic Programs}
Consider the following quadratic optimization problems with mixed integer variables:
\begin{eqnarray*}
(P_M) & \inf_{x \in \mathbb{R}^n} & x^TAx+b^Tx+c \\
& \mbox{  s.t.  }& a_j^Tx=b_j, \, j=1,\ldots,m, \\
& & x_i \in \{0,1\}, \, i \in B, \\
& & x \ge 0,
\end{eqnarray*}
where $B=\{1,\ldots,s\}$ with $s \le n$. Throughout this section, we always assume that the feasible set of $(P_M)$ is nonempty. The quadratic optimization problems with mixed integer variables is a broad and difficult class of
quadratic optimization problem which includes several well-known NP-hard problems such as the knapsack problems.
\medskip

\noindent{\bf Hidden Convexifiability of $(P_M)$}. We say that the discrete problem $(P_M)$ admits hidden convexifiability
whenever its equivalent continuous quadratic program reformulation is convexifiable,

In the celebrated paper of \cite{Bu09}, copositive representation and exact completely positive relaxation results were presented for quadratic optimization problems with mixed integer variables $(P_M)$
 under the following key regularity assumption
 \[
(RA) \ \ \ a_j^Td=b_j, j=1,\ldots,m, \, d \in \mathbb{R}^n_+ \  \Rightarrow \ 0 \le d_i \le 1, \, i \in B.
\]
It was demonstrated in \cite{Bu09} that the regularity assumption (RA) can always be satisfied by introducing slack variables. In this section, we establish that this class of mixed integer programs under the same regularity
condition assumed in \cite{Bu09} admits hidden convexifiability and consequently enjoys the zero duality gap property.

We first note that the problem $(P_M)$  can be equivalently rewritten as
\begin{eqnarray*}
 & \min_{x \in \mathbb{R}^n} & x^TAx+b^Tx+c \\
& \mbox{  s.t.  }& a_j^Tx-b_j = 0, \, j=1,\ldots,m, \\
%& & -a_j^Tx+b_j \le 0, \, j=1,\ldots,m, \\
& & (a_j^Tx)^2-b_j^2 = 0, \, j=1,\ldots,m, \\
%& & -(a_j^Tx)^2+b_j^2 \le 0, \, j=1,\ldots,m, \\
& & x_i(x_i-1) = 0, \, i \in B \\
%& & -x_i(x_i-1) \le 0, \, i \in B \\
& & x\ge 0,
\end{eqnarray*}
which can be further rewritten as the following quadratic programming problems with quadratic inequality constraints:
\begin{eqnarray*}
(P_D) & \min_{x \in \mathbb{R}^n} & x^TAx+b^Tx+c \\
& \mbox{  s.t.  }& a_j^Tx-b_j \le 0, \, j=1,\ldots,m, \\
& & -a_j^Tx+b_j \le 0, \, j=1,\ldots,m, \\
& & (a_j^Tx)^2-b_j^2 \le 0, \, j=1,\ldots,m, \\
& & -(a_j^Tx)^2+b_j^2 \le 0, \, j=1,\ldots,m, \\
& & x_i(x_i-1) \le 0, \, i \in B \\
& & -x_i(x_i-1) \le 0, \, i \in B \\
& & x\ge 0.
\end{eqnarray*}
Let $f(x)=x^TAx+b^Tx+c$ and
\[
g_j(x)= \left\{ \begin{array}{ll}
        a_j^Tx-b_j, &  j=1,...,m, \\
        -a_{j-m}^Tx+b_{j-m}, & j=m+1,\ldots,2m, \\
        (a_{j-2m}^Tx)^2-b_{j-2m}^2, &  j=2m+1,...,3m, \\
        -(a_{j-3m}^Tx)^2+b_{j-3m}^2, &  j=3m+1,...,4m, \\
        x_{j-4m}(x_{j-4m}-1), & j=4m+1,...4m+s, \\
        -x_{j-4m-s}(x_{j-4m-s}-1), & j=4m+s+1,...4m+2s.
\end{array} \right.
 \]

 Define a dual problem associated with $(P_M)$ as follows
\begin{equation} (D_M) \quad \sup \limits_{u} \Theta(u)\quad \mbox{s.t.}\ \,  u\in \R^{m+1}_+,\end{equation}
where  $\Theta(u)$ is given by $\Theta(u):=\inf\limits_{x\in \R^n_+}L(x,u)$ with $L(x,u):=f(x)+\sum\limits_{j=1}^{4m+2s} u_jg_j(x)$, $f(x)=x^TAx+b^Tx+c,$  and
\[
g_j(x)= \left\{ \begin{array}{ll}
        a_j^Tx-b_j, &  j=1,...,m, \\
        -a_{j-m}^Tx+b_{j-m}, & j=m+1,\ldots,2m, \\
        (a_{j-2m}^Tx)^2-b_{j-2m}^2, &  j=2m+1,...,3m, \\
        -(a_{j-3m}^Tx)^2+b_{j-3m}^2, &  j=3m+1,...,4m, \\
        x_{j-4m}(x_{j-4m}-1), & j=4m+1,...4m+s, \\
        -x_{j-4m-s}(x_{j-4m-s}-1), & j=4m+s+1,...4m+2s.
\end{array} \right.
 \]
The problem $(D_M)$ is indeed the semi-Lagrange dual of the equivalent reformulated problem $(P_D)$.

%As an immediate corollary, we now examine zero duality gap property between $(P_M)$ and the dual problem $(D_M)$.

%
%  We note that, under the following key regularity assumption
%  \[
% (RA) \ \ \ a_j^Td=b_j, j=1,\ldots,m, \, d \in \mathbb{R}^n_+ \  \Rightarrow \ 0 \le d_i \le 1, \, i \in B,
% \]
% Burer \cite{Bu09} showed that the complete positive relaxation is exact, in the sense that $\inf(P_M)=\inf(CP_D)$ provided the key regularity assumption (RA) holds. It was also

We now show that the problem $(P_M)$ admits hidden convexifiability under suitable conditions and consequently zero duality holds for $(P_M)$.  Our method of proof, in part, employs  the proof techniques utilized in Burer's paper \cite{Bu09}.
% As a consequence, we derive exact copositive relaxation for $(P_M)$ in the sense that $\inf(P_M)=\inf(CP_D)=\sup(CP_D^*)$ by assuming condition (RA) and an additional mild assumption.

\begin{theorem}{\bf (Hidden convexfiability and zero duality gaps)} \label{thm:mixed_integer}
For problem $(P_M)$ and its equivalent reformulation $(P_D)$, let $$\mathcal{A}_{P_D}:=\big\{\big(g_1(x),..., g_{4m+2s}(x), f(x)\big) : x\in\R^n_+\big\}+ \mathbb{R}^{4m+2s+1}_{+}.$$ Suppose that the regularity assumption (RA) holds and
$$\{d \in \mathbb{R}^n_+: d^TAd \le 0, a_j^Td=0, j=1,\ldots,m, d_i=0, 1 \le i \le s\}=\{0_{\R^n}\}.$$ Then,
\[
\big(\{0_{{\mathbb{R}^{4m+2s}}}\} \times \mathbb{R}\big) \cap {\rm cl \, conv} \mathcal{A}_{P_D}= \big(\{0_{{\mathbb{R}^{4m+2s}}}\} \times \mathbb{R}\big) \cap  \mathcal{A}_{P_D},
\]
and problem $(P_M)$ admits hidden convexifiability. Moreover, $\inf(P_M)=\sup(D_M).$

%\noindent (ii) In particular, if (RA) holds and $$\{d \in \mathbb{R}^n_+: d^TAd \le 0, a_j^Td=0, j=1,\ldots,m, d_i=0, 1 \le i \le s\}=\{0_{\R^n}\}$$ then, $\inf(P_M)=\sup(D_M).$
%$\big(0_{\mathbb{R}^{2m+s}}, \inf(P_D)-\epsilon\big) \notin {\rm  cl \, conv} (\mathcal{A}_{P_D})$.
%\begin{itemize}
%\item[{\rm (i)}] $\big(0_{\mathbb{R}^{2m+s}}, \inf(P_D)-\epsilon\big) \notin {\rm   conv} (\mathcal{A}_{P_D})$ for each $\epsilon>0$, and $\inf(P_D)=\inf(CP_D)$;
%\item[{\rm (ii)}] Suppose, in addition, that  In particular, one has  $$\inf(P_D)=\inf(CP_D)=\sup(CP_D^*).$$
%\end{itemize}
\end{theorem}

\begin{proof}   We first observe that $\big(\{0_{{\mathbb{R}^{4m+2s}}}\} \times \mathbb{R}\big) \cap {\rm cl \, conv} \mathcal{A}_{P_D} \supseteq \big(\{0_{{\mathbb{R}^{4m+2s}}}\} \times \mathbb{R}\big) \cap  \mathcal{A}_{P_D}$ always holds. To see the reverse inclusion, let
$(0_{{\mathbb{R}^{4m+2s}}},\mu) \in {\rm cl \, conv} \mathcal{A}_{P_D}$.  It suffices to show that  $(0_{{\mathbb{R}^{4m+2s}}},\mu) \in \mathcal{A}_{P_D}$.
Denote
\begin{equation}
 H:=\begin{pmatrix} c& b^T/2\\
b/2& A
\end{pmatrix}, \end{equation}
and
\begin{equation}  H_j=\left\{ \begin{array}{ll}
                          \begin{pmatrix} -b_j & a_j^T/2\\
a_j/2& 0_{n \times n}
\end{pmatrix}, & j=1,...,m, \\
\begin{pmatrix} b_{j-m} & -a_{j-m}^T/2\\
-a_{j-m}/2& 0_{n \times n}
\end{pmatrix},\, & j=m+1,...,2m, \\
\begin{pmatrix} -b_{j-2m}^2 & 0^T\\
0& a_{j-2m}a_{j-2m}^T
\end{pmatrix},\, & j=2m+1,...,3m,  \\
\begin{pmatrix} b_{j-3m}^2 & 0^T\\
0& -a_{j-3m}a_{j3m}^T
\end{pmatrix},\,& j=3m,...,4m, \\
\begin{pmatrix} 0 & -e_{j-4m}^T/2\\
-e_{j-4m}/2& e_{j-4m}e_{j-4m}^T
\end{pmatrix} ,\,& j=4m+1,...,4m+s \\
\begin{pmatrix} 0 & e_{j-4m-s}^T/2\\
e_{j-4m-s}/2& -e_{j-4m-s}e_{j-4m-s}^T
\end{pmatrix} ,\, & j=4m+s+1,...,4m+2s. \end{array}\right.\end{equation}
Here $e_j$ is the unit vector whose $j$th element is one and the other elements are all zero.
% Define
% $$\overline{H}_j=\left\{\begin{array}{ll}
%                       H_j, & j=1,\ldots,m, \\
%                       -H_{j-m}, & j=m+1,\ldots,2m \\
%                       H_{j-m}, & j=2m+1,\ldots, 3m, \\
%                       H_{j-2m}, & j=3m+1,\ldots, 4m, \\
%                       H_{j-2m}, & j=4m+1,\ldots,4m+s \\
%                       -H_{j-2m-s}, & j=4m+s+1,\ldots,4m+2s
%                       \end{array}
%  \right.$$
Direct verification shows that $g_j(x)=\left(\begin{array}{c}
                                                                       1 \\
                                                                       x
                                                                      \end{array} \right)^T H_j\left(\begin{array}{c}
                                                                       1 \\
                                                                       x
                                                                      \end{array} \right)$, $j=1,\ldots,4m+2s$. Denote the trace of a $(p \times p)$ matrix $A$ by ${\rm Tr}(A)$ and
                                                                      recall that $a^TAa={\rm Tr}\big(A (aa^T\big))$ for any $a \in \mathbb{R}^p$. It follows that
\begin{eqnarray*}
\mathcal{A}_{P_D} &= & \{(z^TH_1z,\ldots,z^TH_{4m+2s}z,z^THz): z=\left(\begin{array}{c}
                                                                       1 \\
                                                                       x
                                                                      \end{array} \right), \
 x \in \mathbb{R}^{n}_+\}+  \mathbb{R}^{4m+2s+1}_{+}  \\
& = & \{({\rm Tr} (H_1X),\ldots,{\rm Tr} (H_{4m+2s}X),{\rm Tr} (HX)): X=\left(\begin{array}{c}
                                                                       1 \\
                                                                       x
                                                                      \end{array} \right)\left(\begin{array}{c}
                                                                       1 \\
                                                                       x
                                                                      \end{array} \right)^T, \,  x\in \mathbb{R}^{n}_+\}\\
                                                                      & & +  \mathbb{R}^{4m+2s+1}_{+} \\
& \subseteq  & \{({\rm Tr} (H_1X),\ldots,{\rm Tr} (H_{4m+2s}X),{\rm Tr} (HX)): X \in M \}+ \mathbb{R}^{4m+2s+1}_{+},
\end{eqnarray*}
where $M=\{X \in  \mathcal{C}: X_{11}=1\}$ and $\mathcal{C}$ is the completely positive cone, given by
$\mathcal{C}:={\rm conv}\{\tilde x \tilde x^T : \tilde x\in \R^{n+1}_+\}$.

Let
\[
K=\{({\rm Tr} (H_1X),\ldots,{\rm Tr} (H_{4m+2s}X),{\rm Tr} (HX)): X \in M \}+\mathbb{R}^{4m+2s+1}_{+}.
\]
Clearly $K$ is a convex set and $\mathcal{A}_{P_D}\subseteq K$. We claim that $K$ is closed. Granting this, we have  ${\rm cl \, conv} \, \mathcal{A}_{P_D} \subseteq K$.  %To see {\rm (i)}, we now claim that  $\big(0_{\mathbb{R}^{2m+s}}, \mu \big) \notin K$.  To see this claim, we proceed by the method of contradiction and assume  that
 So, $\big(0_{\mathbb{R}^{4m+2s}}, \mu \big) \in K$. Then, there exists $X \in M$ such that
\[
 {\rm Tr} (HX) \le \mu \mbox{ and } {\rm Tr }(H_jX) \le 0,  \ j=1,\ldots,4m+2s.
\]
As $X \in M$, one can write $X=\begin{pmatrix} 1 & w^T \\
w& W
\end{pmatrix}$. It then follows that
\begin{eqnarray}\label{eq:00}
\left \{\begin{array}{ll}
\mu \ge {\rm Tr}(HX)=c + b^Tw+{\rm Tr}(AW) & \\
0=a_j^Tw-b_j, & j=1,\ldots,m, \\
0=-b_{j-m}^2 +a_{j-m}^TWa_{j-m}, &j=m+1,\ldots,2m, \\
0=-e_{j-2m}^Tw+e_{j-2m}^TWe_{j-2m}, & j=2m+1,\ldots,2m+s.
\end{array}\right.
\end{eqnarray}
Moreover, as $X \in M \subseteq \mathcal{C}$, we can write
\[
X=\sum_{l=1}^{K} \left(\begin{array}{c}
\lambda_l \\
y_l
\end{array}
\right)\left(\begin{array}{c}
\lambda_l \\
y_l
\end{array}
\right)^T \mbox{ with } \lambda_l \ge 0, y_l \in \mathbb{R}^{n}_+.
\]
So,
\begin{equation}\label{eq:use1}
\sum_{l=1}^K \lambda_l^2 = 1, \, \sum_{l=1}^K \lambda_l y_l=w \mbox{ and } \sum_{l=1}^K y_ly_l^T =W.
\end{equation}
%Clearly, $r \ge 0$. If $r=0$, then $\xi_l=0$, $l=1,\ldots,K$ and $w=0$. The last relation in \eqref{eq:00} shows that $y_l=0$, $l=1,\ldots,K$ and hence $W=0$. This contradicts to the first relation in \eqref{eq:00}. Thus, $r>0$.
Then, for all $j=1,\ldots,m$, we have
%\[
%\sum_{l=1}^K \lambda_l^2 = \sum_{l \in K_+} \lambda_l^2=1, \
%\sum_{l=1}^K \lambda_l y_l= \sum_{l \in K_+} \lambda_l y_l=\frac{w}{\sqrt{r}},
%\]
\begin{equation}\label{eq:09}
a_j^T(\sum_{l=1}^K \lambda_l y_l)%=a_j^T(\sum_{l \in K_+} \lambda_l y_l)
=a_j^Tw = b_j.
\end{equation}
This implies that
{\small \[
\big(\sum_{l=1}^K \lambda_l (a_j^T y_l)\big)^2= b_j^2=a_j^TWa_j= a_j^T(\sum_{l=1}^K y_ly_l^T)a_j   =\sum_{l=1}^K(a_j^Ty_l)^2=(\sum_{l=1}^K \lambda_l^2) \sum_{l=1}^K(a_j^Ty_l)^2,
\]}
where the second equality is from the third relation of \eqref{eq:00}; the third and the fifth equality follows from \eqref{eq:use1}.
Therefore, the equality case of the Cauchy-Schwartz inequality implies that
\[
a_j^Ty_l=\delta_j \lambda_l, \ l=1,\ldots,K, \, j=1,\ldots,m.
\]
This together with \eqref{eq:09} and $\sum_{l=1}^K \lambda_l^2=1$ shows that
$\delta_j=b_j, \, j=1,\ldots,m$. Denote %$\lambda_l=\frac{\xi_l}{\sqrt{r}}$ and
$K_+=\{l: \lambda_l>0\} \mbox{ and } K_0=\{l: \lambda_l=0\}$.
Define $x_l=\frac{y_l}{\lambda_l}$ for all $l \in K_+$. Then, for all $j=1,\ldots,m$, we have
\begin{equation}\label{eq:a_j}
a_j^Tx_l=\frac{a_j^Ty_l}{\lambda_l}=b_j, \mbox{ for all } l \in K_+
 \mbox{ and }
a_j^Ty_l=0 \mbox{ for all }  \in K_0.
\end{equation}
From the assumption (RA) and $a_j^Ty_l=0$ for all $l \in K_0$ and $j=1,\ldots,m$, we see that $(y_l)_i=0$ for all $l \in K_0$ and for all $i \in B$ (Otherwise, there exists $l_0 \in K_0$ and $i_0 \in B$ such that $(y_{l_0})_{i_0} \neq 0$. As $y_{l_0} \in \R^n_+$,  $(y_{l_0})_{i_0} > 0$. Take any feasible point $x$ of $(P_M)$. Then, $a_j^T(x+t y_{l_0})=b_j$ for all $t \ge 0$ and $j=1,\ldots,m$. Note that $x+t y_{l_0} \in \mathbb{R}^n_+$ and $(x+t y_{l_0})_{i_0}>1$ when $t$ is large enough. This contradicts with assumption (RA)).
So,
\begin{eqnarray}\label{eq:991}
X&=&\sum_{l \in K_+} \left(\begin{array}{c}
\lambda_l \\
y_l
\end{array}
\right)\left(\begin{array}{c}
\lambda_l \\
y_l
\end{array}
\right)^T +\sum_{l \in K_0} \left(\begin{array}{c}
0 \\
y_l
\end{array}
\right)\left(\begin{array}{c}
0 \\
y_l
\end{array}
\right)^T \nonumber \\
& = & \sum_{l \in K_+} \lambda_l^2  \left(\begin{array}{c}
1 \\
x_l
\end{array}
\right)\left(\begin{array}{c}
1 \\
x_l
\end{array}
\right)^T+\sum_{l \in K_0} \left(\begin{array}{c}
0 \\
y_l
\end{array}
\right)\left(\begin{array}{c}
0 \\
y_l
\end{array} \right)^T.
\end{eqnarray}
Now we see that $(x_l)_i \in \{0,1\}$ for all $l \in K_+$ and for all $i \in B$. To see this, from the last relation of \eqref{eq:00},
$w_i=W_{ii}, \ i \in B$.
So, for all $i \in B$
\[
\sum_{l \in K_+} \lambda_l^2  (x_l)_i=w_i=W_{ii}=\sum_{l \in K_+} \lambda_l^2 \,  (x_l)_i^2+\sum_{l \in K_0} (y_l)_i^2=\sum_{l \in K_+} \lambda_l^2\,  (x_l)_i^2,
\]
where the third equality is from \eqref{eq:991} and $X=\begin{pmatrix} 1 & w^T \\
w& W
\end{pmatrix}$, and the last equality follows by the fact that $(y_l)_i=0$ for all $l \in K_0$ and for all $i \in B$. Thus,
\begin{equation}\label{eq:0_1}
\sum_{l \in K_+} \lambda_l^2  \big((x_l)_i- (x_l)_i^2\big)=0.
\end{equation}
Fix $l \in K_+$. Then, \eqref{eq:a_j} gives us that $a_j^Tx_l=b_j$, $j=1,\ldots,m$.
This together with assumption (RA) implies that $0 \le (x_l)_i \le 1$, $i \in B$.
It follows from \eqref{eq:0_1} that, for each $l \in K_+$, $(x_l)_i \in \{0,1\}$ for all $i \in B$. In particular, for all $l \in K_+$, $x_l$  are feasible for $(P_D)$.

Let $\overline{x}={\rm argmin}_{ l \in K_+}\{c+b^Tx_l + x_l^TAx_l\}$. Then, one has $\overline{x}$ is feasible for $(P_D)$, that is $g_j(\bar x) \le 0$, $j=1,\ldots,4m+2s$. Moreover,
\begin{eqnarray*}
\mu \ge {\rm Tr}(HX)= c+b^Tw+{\rm Tr}(AW) & = & (\sum_{l \in K_+} \lambda_l^2 c) + b^T(\sum_{l \in K_+} \lambda_l y_l) + {\rm Tr}\big(A (\sum_{l \in K_+}(y_ly_l^T)\big) \\ %\sum_{l \in K_+} \lambda_l^2  \left(\begin{array}{c}
%1 \\
%x_l
%\end{array}
%\right)^T\left(\begin{array}{cc}
%c-\inf(P_D)& b^T/2\\
%b/2& A
%\end{array}
%\right) \left(\begin{array}{c}
%1 \\
%x_l
%\end{array}
%\right)\\
 & = & \sum_{l \in K_+} \lambda_l^2 [c+b^Tx_l+ x_l^TAx_l] \\
%& = & c +  \sum_{l \in K_+} \lambda_l^2 \big(b^Tx_l +  x_l^TAx_l\big) \\
& \ge & c +  b^T\overline{x} +  \overline{x}^TA\overline{x}=f(\bar x).
\end{eqnarray*}
This shows that $\big(0_{\mathbb{R}^{4m+2s}}, \mu \big) \in \mathcal{A}_{P_D}$. So,  the desired inclusion $\big(\{0_{{\mathbb{R}^{4m+2s}}}\} \times \mathbb{R}\big) \cap {\rm cl \, conv} \mathcal{A}_{P_D} \subseteq \big(\{0_{{\mathbb{R}^{4m+2s}}}\} \times \mathbb{R}\big) \cap  \mathcal{A}_{P_D}$ holds.
Thus, we see that $(P_D)$ is convexfiable, that is,  \[
\big(\{0_{{\mathbb{R}^{4m+2s}}}\} \times \mathbb{R}\big) \cap {\rm cl \, conv} \mathcal{A}_{P_D}= \big(\{0_{{\mathbb{R}^{4m+2s}}}\} \times \mathbb{R}\big) \cap  \mathcal{A}_{P_D},
\]
and so, problem $(P_M)$ enjoys hidden convexifiability.

%In particular, for all  $\epsilon>0$ and for all $X \in M=\{x \in  \mathcal{C}: X_{11}={\rm Tr}(J_0X)=1\}$ with ${\rm Tr}(H_iX)=0$, $i=1,\ldots,2m+s$,
%\[
%{\rm Tr}(HX) \ge \inf(P_D)-\epsilon.
%\]
% This implies that $\inf(CP_D) \ge \inf (P_D)-\epsilon$ for all $\epsilon>0$. By letting $\epsilon \rightarrow 0$, we have $\inf(CP_D) \ge \inf (P_D)$. As the reverse inequality always holds
% by construction, we see that $\inf(P_D)=\inf(CP_D)$.
%Moreover, as ${\rm conv} \, \mathcal{A}_{P_D} \subseteq K$, we also see that
%$\big(0_{\mathbb{R}^{2m+s}}, \inf(P_D)-\epsilon\big) \notin {\rm conv} \, \mathcal{A}_{P_D}$.
%This establishes {\rm (i)}.

%[Proof of {\rm (ii)}] Suppose, in addition, that

We now justify our claim that $K$ is closed. To see this, let
$$(u_1^{(k)},\ldots,u_{4m+2s}^{(k)},r^{(k)}) \in K \ \mbox{with} \ (u_1^k,\ldots,u_{4m+2s}^k,r^k) \rightarrow (u_1,\ldots,u_{4m+2s},r).$$ Then, there exist $X_k \in M$ such that
\begin{equation}\label{eq:990}
{\rm Tr}(H_j X_k)\le u_j^{(k)}, \ j=1,\ldots,4m+2s \mbox{ and } {\rm Tr}(H X_k) \le r^{(k)}.
\end{equation}
% In particular, one has
% \[
% a_j^Tx_k-b_j=u_j^{(k)}, j=1,\ldots,m, \ (x_{k})_i\big(({x_{k}})_i-1\big)=u_{2m+i}^{(k)}, i=1,\ldots,s, \mbox{ and } x_k^TAx_k+b^Tx_k+c \le r^{(k)}.
% \]
This shows that $\{X_k\}$ is bounded. (Otherwise, by passing to subsequence if necessary, we can assume that $\|X_k\| \rightarrow \infty$. Then, by passing to subsequence, we can further
assume that $\frac{X_k}{\|X_k\|} \rightarrow \overline{X}$ with $\overline{X} \in \mathcal{C} \backslash \{0\}$ and $\overline{X}_{11}=0$. Dividing \eqref{eq:990} by $\|X_k\|$ and passing to the limit,
one has
\begin{equation}\label{eq:9991}
{\rm Tr}(H_j \overline{X}) \le 0, \ j=1,\ldots,4m+2s \mbox{ and } {\rm Tr}(H \overline{X}) \le 0.
\end{equation}
As $\overline{X} \in \mathcal{C}$, there exists $K \in \mathbb{N}$ such that
\[
\overline{X}=\sum_{l=1}^{K} \left(\begin{array}{c}
\lambda_l \\
d_l
\end{array}
\right)\left(\begin{array}{c}
\lambda_l \\
d_l
\end{array}
\right)^T \mbox{ with } \lambda_l \ge 0, d_l \in \mathbb{R}^{n}_+.
\]
Note that $\overline{X}_{11}=0$. We have
\[
\overline{X}=\left(\begin{array}{cc}
0 & 0 \\
0 & \sum_{l=1}^K d_ld_l^T
\end{array}\right) \mbox{ with }  d_l \in \mathbb{R}^{n}_+.
\]
As $\overline{X} \neq 0$, by decreasing $K$ if necessary, we can assume without loss of generality that $d_l \neq 0$ for all $l=1,\ldots,K$.
It then follows from \eqref{eq:9991} that
\[
\sum_{l=1}^K (a_j^Td_l)^2=0, j=1,\ldots,m,  \sum_{l=1}^K (d_l)_i^2=0, 1 \le i \le s, \mbox{ and } \sum_{l=1}^K d_l^TAd_l \le 0.
\]
The last relation entails that there exists $l _0 \in \{1,\ldots,K\}$ such that $d_{l_0}^TAd_{l_0} \le 0$. This together with
$a_j^Td_{l_0}=0$, $j=1,\ldots,m$ and $(d_{l_0})_i=0$, $i=1,\ldots,s$, implies that
\[
d_{l_0} \in \{d \in \mathbb{R}^n_+: d^TAd \le 0, a_j^Td=0, j=1,\ldots,m, d_i=0, 1 \le i \le s\}
\]
From our assumption, one has $d_{l_0}=0$. This contradicts $d_l \neq 0$ for all $l=1,\ldots,K$, and so, $\{X_k\}$ is bounded. Passing to subsequence, we can assume that
$X_k \rightarrow \bar X$. Letting $k \rightarrow \infty$ in \eqref{eq:990}, one has
\[
{\rm Tr}(H_j \bar X) \le u_j, \ j=1,\ldots,4m+2s \mbox{ and } {\rm Tr}(H \bar X) \le r.
\]
So, $(u_1,\ldots,u_{4m+2s},r) \in K$, and hence $K$ is closed.

Finally, applying Theorem \ref{lemma2.1} together with hidden convexifiability gives us immediately that
$\inf(P_M)=\inf(P_D)=\sup(D_M)$.

\hfill \end{proof}

\section{Application to Robust Mixed Integer QPs}
As an application of the results of the previous sections, we now consider robust mixed integer quadratic optimization problem with objective data uncertainty and establish
zero duality gap for robust mixed integer quadratic optimization problems. Robust mixed integer quadratic optimization problems under data uncertainty appear in a variety of application areas (cf. \cite{robust_book,Bersimas}). These problems are generically NP-hard \cite{Bersimas}.

As we see later, we establish that, a class of robust mixed integer quadratic optimization problem with objective data uncertainty admits hidden convexfiability, and so,
admits zero duality gap under mild assumptions.  We note that  exact completely positive relaxation results have been achieved for stochastic linear optimization problem with mixed integer constraints under  distributional data uncertainty \cite{Teo}. Here, different to \cite{Teo}, we consider \textit{deterministic quadratic optimization problem} with commonly used data uncertainty \cite{robust_book}, and we obtain a gap-free duality result under suitable conditions.
%in \cite{Burer}. Here, different to \cite{Burer}, we consider robust convex quadratic problems and our data uncertainty set allows mixed-integer variables.

%problems are generically NP-hard \cite{Adaptive}.
%In this section, we consider robust multi-stage  convex quadratic optimization problem with quadratic constraints (QCQP) with objective data uncertainty. Robust multi-stage QCQPs under data uncertainty appears in a variety of application areas (cf. \cite{robust_book,Adaptive,Bersimas}). These
%problems are generically NP-hard \cite{Adaptive}. We refer the reader to \cite{Tutorial} for a comprehensive review of recent results in robust multi-stage optimization.
%
%
%As we see in this Section, we establish that, a class of  two-stage robust quadratic optimization problems admit an exact copositive relaxation under mild assumptions.  We note that, very recently, exact copositive relaxation results have been achieved for two-staged robust linear optimization with convex and compact data uncertainty set
%in \cite{Burer}. Here, different to \cite{Burer}, we consider robust convex quadratic problems and our data uncertainty set allows mixed-integer variables.

Consider the following robust mixed integer quadratic optimization problem with objective data uncertainty
\begin{eqnarray*}
(RP) & \displaystyle \inf_{x \in \R^n} & \max_{(c,A) \in \mathcal{U} \times \mathcal{V}}\{x^TAx+c^Tx\} \\
& \mbox{  s.t.  } & a_j^Tx=b_j, j=1,\ldots,m,\\
& &  x \ge 0, \, x_i \in \{0,1\},  i \in B.
\end{eqnarray*}
where  $a_j \in \mathbb{R}^n, \ b_j \in \mathbb{R}, j=1,\ldots,m$ and $B=\{1,\ldots,s\}$ with $s \le n$. Here $\mathcal{U}$  is a commonly used compact polyhedral data uncertainty set given by
\[
\mathcal{U}=\bigg\{c_0+\sum_{l=1}^L \xi_l \, c_l: \xi=(\xi_1,\ldots,\xi_L)\in \R^L, \xi \in {\rm conv}\{\xi^{(1)},\ldots,\xi^{(q)}\}\bigg\}
\]
with $\xi^{(k)} \in \R^L$, $k=1,\ldots,q$, and $\mathcal{V}$ is the spectral norm uncertainty set
\[
\mathcal{V}=\bigg\{A_0+ V: \|V\|_{\rm spec} \le \rho\bigg\},
\]
where $\|A\|_{\rm spec}=\sqrt{\lambda_{\max}(A^TA)}$ for any symmetric $(n \times n)$ matrix $A$ and for any symmetric matrix $M$, $\lambda_{\max}(M)$ is the largest eigenvalue of $M$.
Throughout this section, we will assume that the regularity assumption (RA) holds and
\begin{equation}\label{eq:assumption2}
\{d \in \mathbb{R}^n_+: a_j^Td=0, j=1,\ldots,m, d_i=0, 1 \le i \le s\}=\{0_{\R^n}\}.
\end{equation}
Denote the feasible set of (RP) by $C$. We note that the assumption \eqref{eq:assumption2} is equivalent to the fact that the feasible set $C$  is a compact set, and so,
\begin{equation} \label{eq:M}
M:=\max\{\max_{x \in C, 1 \le k \le q}\{c_0^Tx+\sum_{l=1}^L \xi_l^{(k)} \, c_l^Tx\},0\}<+\infty.
\end{equation}

We now specify a dual problem associated with problem (RP). To do this,  define
\[
W=\left(\begin{array}{cccc}
A_0+\rho I_n & 0_{\R^n} & 0_{\R^n} & 0_{n \times (q+2)} \\
0_{\R^n}^T & 0 & 0 & 0_{\R^{q+2}}^T \\
0_{\R^n}^T & 0 & 0 & 0_{\R^{q+2}}^T \\
0_{(q+2) \times n} & 0_{\R^{q+2}} & 0_{\R^{q+1}} & 0_{(q+2) \times (q+2)}
\end{array} \right), \ \
w=\left(\begin{array}{c}
0_{\R^n} \\
1 \\
-1 \\
0_{\R^{q+2}}
\end{array}\right)
\]
\begin{equation}\label{overline_a}
\overline{a}_j=\left\{\begin{array}{lll}
\left(\begin{array}{c}
a_j \\
0 \\
0 \\
0_{\R^{q+2}}
\end{array}\right)
 & \mbox{ if } & j=1,\ldots,m \\
 \left(\begin{array}{c}
c_0+\sum_{l=1}^L \xi_l^{(j-m)} \, c_l \\
-1 \\
1 \\
e_{j-m}
\end{array}\right)
  & \mbox{ if } & j=m+1,\ldots,m+q \\
   \left(\begin{array}{c}
0_{\R^n}  \\
1 \\
0 \\
e_{q+1}
\end{array}\right)  & \mbox{ if } & j=m+q+1 \\
%\left(\begin{array}{c}
%0_{\R^n}  \\
%1 \\
%0 \\
%e_{q+2}
%\end{array}\right)  & \mbox{ if } & j=m+q+2 \\
 \left(\begin{array}{c}
0_{\R^n}  \\
1 \\
-1 \\
e_{q+2}
\end{array}\right)  & \mbox{ if } & j=m+q+2
\end{array} \right.
\end{equation}
where, $I_n$ is the $n\times n$ identity matrix, for each $j=1,\ldots,q+2$, ${e}_j \in \mathbb{R}^{q+2}$ is the vector whose $j$th element is one and the others are zero,
and
\begin{equation}\label{overline_b}
\overline{b}_j=\left\{\begin{array}{lll}
b_j & \mbox{ if } & j=1,\ldots,m, \\
0  & \mbox{ if } & j=m+1,\ldots,m+q, \\
-M & \mbox{ if } & j=m+q+1, m+q+2,
\end{array} \right. ,
\end{equation}
where the constant $M$ is given as in \eqref{eq:M}. Denote $l(m,q)=m+q+2$. We now define a dual problem associated with $(RP)$ as follows
\begin{equation}\label{DRP} (D_{RP}) \quad \sup \limits_{u} \Theta(u)\quad \mbox{s.t.}\ \,  u\in \R^{4l(m,q)+2s}_+,\end{equation}
where  $\Theta(u)$ is given by $\Theta(u):=\inf\limits_{x\in \R^{n+q+4}_+}L(x,u)$ with $L(x,u):=f(x)+\sum\limits_{j=1}^{4l(m,q)+2s} u_jg_j(x)$, $f(x)=x^TWx+w^Tx,$  and
\[
g_j(x)= \left\{ \begin{array}{ll}
        \overline{a}_j^Tx-\overline{b}_j, &  j=1,...,l(m,q), \\
        -\overline{a}_{j-l(m,q)}^Tx+\overline{b}_{j-l(m,q)}, & j=l(m,q)+1,\ldots,2l(m,q), \\
        (a_{j-2l(m,q)}^Tx)^2-b_{j-2l(m,q)}^2, &  j=2l(m,q)+1,...,3l(m,q), \\
        -(a_{j-3l(m,q)}^Tx)^2+b_{j-3l(m,q)}^2, &  j=3l(m,q)+1,...,4l(m,q), \\
        x_{j-4l(m,q)}(x_{j-4l(m,q)}-1), & j=4l(m,q)+1,...4l(m,q)+s, \\
        -x_{j-4l(m,q)-s}(x_{j-4l(m,q)-s}-1), & j=4l(m,q)+s+1,...4l(m,q)+2s.
\end{array} \right.
 \]

{Below, we establish a zero duality gap result for robust mixed integer quadratic programming problems. We achieve this by
identifying hidden convexifiablity of the robust mixed integer quadratic programming problems}. Importantly, the dual problem $(D_{RP})$ can equivalently be reformulated as a copositive programming problem (see \cite{B15}).
\begin{theorem}{\bf (Robust mixed integer QP: zero duality gaps)}
Suppose that $\{d \in \mathbb{R}^n_+: a_j^Td=0, \ j=1,\ldots,m, \ d_i=0, \ i=1,\ldots,s\}=\{0_{\R^n}\}$ and assumption (RA) holds.  Then,  $\inf(RP)=\sup(D_{RP})$.
\end{theorem}
\begin{proof}
Note that a linear function attains its maximum over a polytope at an extreme point of the underlying polytope and $\displaystyle \max_{\|V\|_{\rm spec} \le \rho}x^TVx =\rho \|x\|^2$ for all $x \in \R^n$.
The problem (RP) can be equivalently rewritten as
\begin{eqnarray*}
& \displaystyle \min_{x \in \R^n} & x^T(A_0+\rho I_n) x+ \max_{1 \le k \le q}\{c_0^Tx+\sum_{l=1}^L \xi_l^{(k)} \, c_l^Tx\} \\
& \mbox{  s.t.  } & a_j^Tx=b_j, j=1,\ldots,m,\\
& &  x \ge 0, \, x_i \in \{0,1\},  i \in B.
\end{eqnarray*}
which is further equivalent to
\begin{eqnarray*}
& \displaystyle \min_{(x,t) \in \R^n \times \R} & x^T(A_0+\rho I_n) x+ t  \\
& \mbox{  s.t.  } & a_j^Tx=b_j, j=1,\ldots,m,\\
& & c_0^Tx+\sum_{l=1}^L \xi_l^{(k)} \, c_l^Tx \le t, \ k=1,\ldots,q,\\
& &  x \ge 0, \, x_i \in \{0,1\},  i \in B.
\end{eqnarray*}
Now, recall that $M=\max\{\max_{x \in C, 1 \le k \le q}\{c_0^Tx+\sum_{l=1}^L \xi_l^{(k)} \, c_l^Tx\},0\}.$ Then, the problem can be further rewritten as
\begin{eqnarray*}
(AP_0) & \displaystyle \min_{(x,t) \in \R^n \times \R} & x^T(A_0+\rho I_n) x+ t  \\
& \mbox{  s.t.  } & a_j^Tx=b_j, j=1,\ldots,m,\\
& & c_0^Tx+\sum_{l=1}^L \xi_l^{(k)} \, c_l^Tx \le t, \ k=1,\ldots,q,\\
& & t \le M, \\
& &  x \ge 0, \, x_i \in \{0,1\},  i \in B.
\end{eqnarray*}

Letting $t=t_1-t_2$ with $0 \le t_1 \le M$ and $t_2 \ge 0$, and introducing a slack variable $v_k \ge 0$ for each linear inequality constraint
$c_0^Tx+\sum_{l=1}^L \xi_l^{(k)} \, c_l^Tx \le t$, $k=1,\ldots,q$, $t_1 \le M$ and $t \le M$, we see that the robust problem can be rewritten as the following quadratic optimization problem with mixed linear quadratic optimization problem:
\begin{eqnarray*}
(AP_1) & \displaystyle \min_{(x,t_1,t_2,v) \in \R^n \times \R \times \R \times \R^{q+2}} & x^T(A_0+\rho I_n) x+ t_1-t_2  \\
& \mbox{  s.t.  } & a_j^Tx=b_j, j=1,\ldots,m,\\
& & c_0^Tx+\sum_{l=1}^L \xi_l^{(k)} \, c_l^Tx+v_k -(t_1-t_2)=0, \ k=1,\ldots,q,\\
& & t_1 +v_{q+1} = M, \\
%& & t_2 +v_{q+2}=M, \\
& & t_1-t_2+v_{q+2} =M, \\
& & x \ge 0,t_1 \ge 0,t_2 \ge 0,v \ge 0, \, x_i \in \{0,1\},  i \in B.
\end{eqnarray*}
Indeed, for any feasible point $(x,t)$ of $(AP_0)$, one has $(x, \max\{t,0\}, -\min\{t,0\},v)$ is feasible for $(AP_1)$ with the same objective value, where $v_k=t-(c_0^Tx+\sum_{l=1}^L \xi_l^{(k)} \, c_l^Tx)$, $k=1,\ldots,q$, $v_{q+1}=M-\max\{t,0\}$, and %$v_{q+2}=M + \min\{t,0\}$ and
$v_{q+2}=M-t$.
On the other hand, for any feasible point $(x,t_1,t_2,v)$ for $(AP_1)$, $(x,t_1-t_2)$ is feasible for $(AP_0)$ with the same objective value. Thus, we see that $(AP_0)$ and $(AP_1)$ are equivalent and have the same optimal value.

Let $z=(x,t_1,t_2,v) \in \R^n \times \R \times \R \times \R^{q+2}$. Then, this problem can be simplified as
\begin{eqnarray*}
(AP) & \displaystyle \min_{z \in \R^{n+q+4}} & z^TWz+ w^Tz  \\
& \mbox{  s.t.  } & \overline{a}_j^Tz=\overline{b}_j, j=1,\ldots,m+q+2,\\
& & z \ge 0, \, z_i \in \{0,1\},  i \in B.
\end{eqnarray*}
where \[
W=\left(\begin{array}{cccc}
A_0+\rho I_n & 0_{\R^n} & 0_{\R^n} & 0_{n \times (q+2)} \\
0_{\R^n}^T & 0 & 0 & 0_{\R^{q+2}}^T \\
0_{\R^n}^T & 0 & 0 & 0_{\R^{q+2}}^T \\
0_{(q+2) \times n} & 0_{\R^{q+2}} & 0_{\R^{q+1}} & 0_{(q+2) \times (q+2)}
\end{array} \right), \ \
w=\left(\begin{array}{c}
0_{\R^n} \\
1 \\
-1 \\
0_{\R^{q+2}}
\end{array}\right)
\]
and $\overline{a}_j$ and $\overline{b}_j$ are given as in \eqref{overline_a} and \eqref{overline_b} respectively.

We now verify that the assumptions in Theorem \ref{thm:mixed_integer} hold for problem (AP), that is,
\[
 \overline{a}_j^Tu=\overline{b}_j, j=1,\ldots,m+q+2, \, u \in \mathbb{R}^{n+q+4}_+ \  \Rightarrow \ 0 \le u_i \le 1, \, i \in B,
\]
and
{\small $$\{u \in \mathbb{R}^{n+q+4}_+: u^TWu \le 0, \overline{a}_j^Tu=0, j=1,\ldots,m+q+2, u_i=0, 1 \le i \le s\}=\{0_{\R^{n+q+4}}\}.$$}
Indeed, take any $u$ such that
\[
\overline{a}_j^Tu=\overline{b}_j, j=1,\ldots,m+q+4.
\]
Write $u=(d,r_1,r_2,h) \in \mathbb{R}^n \times \mathbb{R} \times \R \times \R^{q+2}$ with $u \in \mathbb{R}^{n+q+4}_+$. Then, the first $m$ equalities implies that
$a_j^Td=b_j, j=1,\ldots,m$, which shows that $0 \le d_i \le 1$, $i=1,\ldots,s$, by the Assumption (RA). Note that $s \le n$, and so, $u_i=d_i$, $i=1,\ldots,s$. Thus, $0 \le u_i \le 1 $, $i=1,\ldots,s$. Moreover, let $u=(d,r_1,r_2,h) \in \mathbb{R}^n \times \mathbb{R} \times \R \times \R^{q+2}$ be such that $u \in \mathbb{R}^{n+q+4}_+$, $u^TWu \le 0, \overline{a}_j^Tu=0$, $j=1,\ldots,m+q+2$, $u_i=0, 1 \le i \le s$. In particular, we see that
\[
a_j^Td=0, j=1,\ldots,m, \ d \ge 0 \mbox{ and } d_i=0, \ i=1,\ldots,s.
\]
So, $d=0_{\R^n}$ by our assumption. Moreover, one has $r_1,r_2 \ge 0$, $h \in \R^{q+2}_+$,\,
\begin{equation}\label{eq:4.0}
- r_1+r_2+ h_{j-m}=(c_0+\sum_{l=1}^L \xi_l^{(j-m)} \, c_l)^Td - r_1+r_2+ h_{j-m} =0, \ j=m+1,\ldots,m+q
\end{equation}
\begin{equation}\label{eq:4.1}
r_1+ h_{q+1}=0,
\end{equation}
and
\begin{equation}\label{eq:4.2}
r_1-r_2+h_{q+2}=0.
\end{equation}
From \eqref{eq:4.1}, $r_1 \ge 0$ and $h \ge 0$, we see that $r_1=0$ and $h_{q+1}=0$. Then, \eqref{eq:4.0} reduces to
$r_2+ h_{j-m}=0$, $j=m+1,\ldots,m+q$, which further implies that $r_2=0$ and $h_j=0$, $j=1,\ldots,q$.
Combining these with \eqref{eq:4.2} gives us that $r_1=r_2=0$ and $h=0_{\R^{q+2}}$, and so, $u=0_{\R^{n+q+4}}$.
%Define
%{\small $H:=\begin{pmatrix} 0 & \frac{w}{2}\\
%\frac{w}{2} & W
%\end{pmatrix}, $
%$$H_j=\begin{pmatrix} -\overline{b}_j & \overline{a}_j^T/2\\
%\overline{a}_j/2& 0_{n \times n}
%\end{pmatrix},\, j=1,...,m+q+1 $$
%$$ H_j=\begin{pmatrix} -\overline{b}_{j-m-q-1}^2 & 0^T\\
%0& \overline{a}_{j-m-q-1}\overline{a}_{j-m-q-1}^T
%\end{pmatrix},\, j=m+q+2,...,2(m+q+1), $$}
%and
%$$ H_j=\begin{pmatrix} 0 & -e_{j-2(m+q+1)}^T/2\\
%-e_{j-2(m+q+1)}/2& e_{j-2(m+q+1)}e_{j-2(m+q+1)}^T
%\end{pmatrix} ,\, j=2(m+q+1)+1,...,2(m+q+1)+s,$$

Now, it follows from Theorem \ref{thm:mixed_integer}  that problem (RP) admits hidden convexifiability and  $\inf(RP)=\sup(D_{RP})$, where $(D_{RP})$ is the corresponding dual problem of $(RP)$ given as in
\eqref{DRP}.
% $$(CP_{RP}^*) \quad\quad  \sup\big\{ y_0 : Z_+(x,y_0,u) \in \mathcal{C}^\star,\, y=(y_0,u)\in \R\times \R^{2(m+q+2)+s}\big\},$$
% where
% \begin{eqnarray*}
% Z_+(x,y_0,u)&=& H+\sum\limits_{j=1}^{2(m+q+2)+s} u_jH_j-y_0J_0, \end{eqnarray*}
% and $H=\begin{pmatrix} 0 & \frac{w^T}{2}\\
% \frac{w}{2} & W
% \end{pmatrix},$ and $H_j$ are given as in \eqref{eq:H}.
Thus, the conclusion follows.  \end{proof}

\section{Appendix: Technical Conditions and Related Links }
In this section we provide sufficient conditions for (strongly) convexifiable nonnegative quadratic programs and present the close links between an exact copositive relaxation result and the zero duality gap property of nonnegative quadratic programs.
\medskip

\noindent{\large \bf Sufficient conditions for strong convexifiability}.  We first examine strong convexifiability of the following uniform nonnegative quadratic program:
 $$(P_2) \quad \begin{array}{rl}& \inf\limits_{x} \ x^TAx+b^Tx+c\\
 &  \mbox{s.t.}\ \, x\in \R^n_+,\,   \alpha_ix^TAx+b_i^Tx+c_i \leq 0,\,  i=0,1,...,m,
 \end{array}
 $$
 where   $A\in S^{n},$  $b,b_i\in \R^n,$ $c,c_i\in \R,$  $i=0,1,...,m,$ $\alpha_i\in \R,$ $i=1,...,m.$ The specific feature of $(P_2)$  is that each  Hessian matrix of the constraint function is different from the one  of the objective function  only by a multiple constant.

 The following result provides some sufficient conditions for strong convexifiability of  $(P_2)$. Interestingly, these sufficient
conditions  are expressed in terms of the original data of the problem, and can be verified efficiently.

  \begin{proposition}{\rm {\bf (Strong convexifiability of nonnegative uniform QPs).}} \label{thm41} For problem $(P_2),$
  let  $f(x)= x^TAx+b^Tx+c$ and $g_i(x)=\alpha_ix^TAx+b_i^Tx+c_i$, $i=0,1,\ldots,m$.
  Suppose that there exist $\gamma \ge 0$ and
 $\beta_i \ge 0$, $i=0,1,\ldots,m$ such that $(\gamma+ \sum_{i=0}^m \alpha_i \beta_i)  A$ is strictly copositive.
  %$\hat x\in \R^n_+$ such that $g_i(\hat x)<0$ for all $i=0, 1,...,m,$.
  Suppose further that one of the following conditions is satisfied:

  $(i)$ $A$  is a positive semidefinite matrix  having some  eigenvector $d\in \R^n_+$ corresponding to a nonzero eigenvalue, with  $(b_i-\alpha_ib)^Td=0$ for all $i=0, 1,...,m;$

  $(ii)$ $-A$ is a positive semidefinite matrix  having some  eigenvector $d\in \R^n_+$ corresponding to a nonzero eigenvalue, with  $(b_i-\alpha_ib)^Td=0$ for all $i=0, 1,...,m;$

  $(iii)$ $A$  has  eigenvectors $d\in \R^n_+$ and $\hat d\in -\R^n_+$   corresponding to a positive eigenvalue  and a negative eigenvalue of $A$, respectively, with  $(b_i-\alpha_ib)^Td=0,$ $(b_i-\alpha_ib)^T\hat d=0$ for all $i=0, 1,...,m.$

    Then, problem $(P_2)$ is strongly convexifiable; thus, the set $$\mathcal{A}_{P_2}:=\big\{\big(g_0(x), g_1(x),..., g_m(x), f(x)\big) : x\in\R^n_+\big\}+ \R^{m+2}_+$$ is closed and convex.
    %In particular, problem $(P_2)$ is convexfiable and  $$\min (P_2)=\sup (D_2).$$

 \end{proposition}
 \noindent{\it Proof.}% Let  Define
 %$$\mathcal{A}_{P_2}:=\big\{\big(g_0(x), g_1(x),..., g_m(x), f(x)\big) : x\in\R^n_+\big\}+ \R^{m+2}_+.$$
 We first show that $\mathcal{A}_{P_2}$
 is closed. To see this, let $(a_0^{k},a_1^{k},\ldots,a_m^k,a_{m+1}^k) \in \mathcal{A}_{P_2}$ be such that
 $(a_0^{k},a_1^{k},\ldots,a_m^k,a_{m+1}^k)  \rightarrow (\bar a_0,\bar a_1,\ldots,\bar a_m,\bar a_{m+1})$ as $k \rightarrow \infty$. Then, there exists $x_k \in \R^n_+$
such that \begin{equation}\label{eq:997}
 g_i(x_k) \le a_i^{k}, i=0,1,\ldots,m, \mbox{ and } f(x_k) \le a_{m+1}^k.
 \end{equation}
We now show that $\{x_k\}$ is bounded. Otherwise, by passing to subsequence, we can assume that $\|x_k\| \rightarrow +\infty$ and $\frac{x_k}{\|x_k\|} \rightarrow d \in \R^n_+\backslash \{0\}$ as $k \rightarrow \infty$ . Let
$\gamma,\beta_i \ge 0$ be such that $(\gamma+ \sum_{i=0}^m \alpha_i \beta_i)  A$ is strictly copositive, and denote
\[
F(x)= \gamma f(x)+ \beta_i \sum_{i=0}^m g_i(x) =  x^T[(\gamma + \sum_{i=0}^m \alpha_i \beta_i)A]x+(\gamma b+\sum_{i=0}^m \beta_i b_i)^Tx+(\gamma c+\sum_{i=0}^m \beta_i c_i).
\]
 Then, due to the strict copositivity of $(\gamma + \sum_{i=0}^m \alpha_i \beta_i)A$,
\[
\lim_{k \rightarrow \infty}\frac{F(x_k)}{\|x_k\|^2}=d^T[(\gamma + \sum_{i=0}^m \alpha_i \beta_i)A]d>0.
\]
On the other hand,
\[
\lim_{k \rightarrow \infty}\frac{F(x_k)}{\|x_k\|^2}= \lim_{k \rightarrow \infty} \frac{\gamma f(x_k)+ \beta_i \sum_{i=0}^m g_i(x_k)}{\|x_k\|^2} \le \lim_{k \rightarrow \infty} 	\frac{\gamma a_{m+1}^k+ \beta_i \sum_{i=0}^m a_i^k}{\|x_k\|^2}=0
\]
This is impossible, and so, $\{x_k\}$ must be bounded. By passing to subsequence, we see that $x_k \rightarrow \bar x \in \R^n_+$. Passing limit in \eqref{eq:997}, we see that
\[
  g_i(\bar x) \le \bar a_i, i=0,1,\ldots,m, \mbox{ and } f(\bar x) \le \bar a_{m+1}.
\]
So, we see that $(\bar a_0,\bar a_1,\ldots,\bar a_m,\bar a_{m+1}) \in \mathcal{A}_{P_2}$. Thus, $\mathcal{A}_{P_2}$ is closed.

 We now show that  the set $\mathcal{A}_{P_2}$
 is convex,  where $f(x)=x^TAx+b^Tx+c$ and  $g_i(x)=\alpha_ix^TAx+b_i^Tx+c_i.$   To do this, let
 $$\Omega:=\big\{\big((b_0-\alpha_0b)^Tx+c_0-\alpha_0 c,...,(b_m-\alpha_m b)^Tx+c_m-\alpha_m c, x^TAx+b^Tx+c\big): x\in \R^n_+\big\}.$$
 Take any $u=(u_0,...,u_m, u_{m+1})\in \Omega,$ $v=(v_0,v_1,...,v_m,v_{m+1})\in \Omega,$ and $\lambda\in (0,1).$ Then there exist $x_u\in \R^n_+$ and $x_v\in \R^n_+$ such that
 $$ u_i= (b_i-\alpha_i b)^Tx_u+c_i-\alpha_i c,\, i=0,1,...,m,\, u_{m+1}=x_u^TAx_u+b^Tx_u+c \mbox{ and }$$
 $$ v_i= (b_i-\alpha_i b)^Tx_v+c_i-\alpha c,\, i=0,1...,m,\,  v_{m+1}=x_v^TAx_v+b^Tx_v+c.$$

{\it Case 1.} Suppose $(i)$ holds.   Let $d\in \R^n_+$ be an eigenvector corresponding to a nonzero eigenvalue of $A$ with   $(b_i-\alpha_ib)^Td=0$ for all $i=0, 1,2,...,m,$
and let $z_t:=z+td$ for $t\in \R.$  Define the function
$\varphi: \R\rightarrow \R$ by
$$\varphi(t):= z_t^TAz_t+b^Tz_t+c\  \mbox{for}\  t\in \R,$$  where $z:=\lambda x_u+(1-\lambda)x_v.$
  Obviously, $\varphi$ is a continuous function. Moreover, since $A$ is positively semidefinite and $d$ is an eigenvector of $A$ corresponding to a nonzero eigenvalue of $A$, we have
\begin{equation*} \begin{array}{rl} \varphi(0)&=z^TAz+b^Tz+c\\
 & \leq \lambda (x_u^TAx_u+b^Tx_u+c)+(1-\lambda) (x_v^TAx_v+b^Tx_v+c)\\
 &=\lambda u_{m+1}+(1-\lambda)v_{m+1},\end{array}\end{equation*}
 and $\lim\limits_{t\rightarrow+\infty}\varphi (t)=+\infty.$  By the intermediate value theorem, there exists
 $t_0\in \R_+$ such that $$z_{t_0}^TAz_{t_0}+b^Tz_{t_0}+c=\varphi(t_0)=\lambda u_{m+1}+(1-\lambda)v_{m+1}.$$
 Note that $(b_i-\alpha_ib)^Td=0$ for all $i=0, 1,2,...,m,$ $t_0\in \R_+,$  and $d\in \R^n_+.$ So  we have  that $z_{t_0}\in \R^n_+,$
 $$z_{t_0}^TAz_{t_0}+b^Tz_{t_0}+c=\lambda u_{m+1}+(1-\lambda)v_{m+1}$$
 and
 $$ (b_i-\alpha_ib)^Tz_{t_0}+c_i-\alpha_i c=\lambda u_i+(1-\lambda)v_i\,\ \mbox{for}\ i=0,1,...,m.$$
This implies $(1-\lambda)u+\lambda v\in \Omega,$ and thus $\Omega$ is convex.

{\it Case 2.} Suppose $(ii)$ holds. Then, according to Case 1, the set
{\small $$\tilde\Omega:=\big\{\big(-(b_0-\alpha_0b)^Tx-(c_0-\alpha_0 c),..,-(b_m-\alpha_m b)^Tx-(c_m-\alpha_m c), -x^TAx-b^Tx-c\big): x\in \R^n_+\big\}$$}
is convex. On the other hand,  $\Omega=-\tilde \Omega.$ Therefore,  $\Omega$ is a convex set.

{\it Case 3.} Suppose $(iii)$ holds.  Let  $d$ and $\hat d$ be two eigenvectors  with the properties given in the condition $(iii).$   Consider the function $\varphi: \R\rightarrow \R$ defined by
$$\varphi(t):= z_t^TAz_t+b^Tz_t+c\  \mbox{for}\  t\in \R,$$
 where  $$z_t:=\begin{cases} z+td\,\  \mbox{for}\  t\in \R_+\\
 z+t\hat d\,\  \mbox{for}\   t\in \R_- \end{cases}$$ and $z:=\lambda x_u+(1-\lambda)x_v.$
 We see that  $\varphi$ is  continuous, $\lim\limits_{t\rightarrow+\infty}\varphi(t)=+\infty$ and   $\lim\limits_{t\rightarrow-\infty}\varphi(t)=-\infty.$ By the intermediate value theorem, there exists   $t_0\in \R$ such that $$z_{t_0}^TAz_{t_0}+b^Tz_{t_0}+c=\varphi(t_0)=\lambda u_{m+1}+(1-\lambda)v_{m+1}.$$  So, similarly to what have been done  in Case 1, the convexity of $\Omega$ follows.

On the other hand, we have  $\mathcal{A}_{P_2}=L(\Omega)+\R^{m+2}_+,$ where
$L: \R^{m+2}\rightarrow \R^{m+2}$ is the linear mapping defined by, for all $y=(y_0,...,y_{m+1})\in \R^{m+2}$, $$L(y):=(y_0+\alpha_0y_{m+1},y_1+\alpha_1y_{m+1},...,y_m+\alpha_my_{m+1},y_{m+1}) \, .$$
Therefore, $\mathcal{A}_{P_2}$ is convex.  So, $\mathcal{A}_{P_2}$ is a closed and convex set.
%From  Theorem \ref{lemma2.1}, we obtain the desired conclusion.
$\hfill\Box$

\medskip

\noindent{\large\bf Links between duality and exact copositive relaxations}. We now present the connections between the semi-Lagrangian duality and an exact copositive relaxation. In particular, we show that our zero duality gap results immediately imply the exactness of copositive relaxations.

We first recall the copositive and completely positive relaxation of $(P_1)$.
The problem $(P_1)$ can be rewritten as follows:
$$\begin{array}{rl}    &\inf\limits_{X\in \mathcal{C}} {\rm Tr}(HX)\\
 &\ \,  \mbox{s.t.}\ \,  {\rm Tr}(H_iX)\leq 0,\,  i=0,1,...,m,\\
 &\quad\quad \  X_{1,1}=1, \, {\rm rank}(X)=1, \end{array}$$
where  $\mathcal{C}:={\rm conv}\{\tilde x \tilde x^T : \tilde x\in \R^{n+1}_+\}$ is the so-called cone of  completely positive matrices,
$$H:=\begin{pmatrix} c& b^T/2\\
b/2& A
\end{pmatrix}\ \mbox{and} \  H_i=\begin{pmatrix} c_i& b_i^T/2\\
b_i/2& A_i
\end{pmatrix},\, i=0,1,...,m.$$

We note that $\mathcal{C}$ is a full-dimensional closed convex pointed cone, and its dual is the so-called copositive cone   defined by
$$\mathcal{C}^\star:=\big\{ Q=Q^T\in \R^{n+1} \ |\  Q\ \mbox{is copositive}\big\}.$$
Recall that a symmetric matrix $Q\in S^n$ is said to be copositive (resp., strictly copositive) if  $x^TQx\geq 0$ for all $x\in \R^n_+$  (resp., $x^TQx> 0$ for all $x\in \R^n_+\backslash\{0\}$).

By removing  the rank one constraint, we get   the completely positive  relaxation of $(P_1)$:
$$(CP_1)\quad\begin{array}{rl}    &\inf\limits_{X\in \mathcal{C}} {\rm Tr}(HX)\\
 &\ \,  \mbox{s.t.}\ \,  {\rm Tr}(H_iX)\leq 0,\,  i=0,1,...,m,\\
 &\quad\quad \  {\rm Tr}(J_0X)=1,\end{array}$$
where  $J_0:=e_0e_0^T$ with $e_0=(1, 0,...,0)\in\R^{n+1}.$
The conic dual of $(CP_1)$ is called the copositive relaxation of $(P_1)$   defined  as follows:
$$(CP_1^*) \quad\quad  \sup\big\{ y_0 : Z_+(y) \in \mathcal{C}^\star,\, y=(y_0,u)\in \R\times \R^{m+1}_+\big\},$$
where $$Z_+(y):=H+\sum\limits_{i=0}^mu_iH_i-y_0J_0=\begin{pmatrix}c+\sum\limits_{i=0}^m u_ic_i-y_0 &  (b+\sum\limits_{i=0}^m u_ib_i)^T/2 \\
 (b+\sum\limits_{i=0}^m u_ib_i)/2&   A+\sum\limits_{i=0}^m u_iA_i
\end{pmatrix}.$$
We say exact copositive relaxation holds if $\inf (P_1)= \sup (CP_1^*)$ and exact completely positive relaxtion holds if $\inf (P_1)= \inf (CP_1)$.

Recently, Bomze~\cite{B15} has shown that the optimal value of $(CP_1^*)$ is equal to the optimal value of the semi-Lagrangian dual problem $(D_1)$
(see Bomze~\cite{B15}):
\begin{equation}  (D_1) \quad \sup \limits_{u} \Theta(u)\quad \mbox{s.t.}\ \,  u\in \R^{m+1}_+,\end{equation}
where  $\Theta(u)$ is given by $\Theta(u):=\inf\limits_{x\in \R^n_+}L(x,u)$ with $L(x,u):=f(x)+\sum\limits_{i=0}^mu_ig_i(x)$ and  $f(x)=x^TAx+b^Tx+c,$  $g_i(x)=x^TA_ix+b_i^Tx+c_i,$ $i=0,1,...,m.$ Clearly, by construction, \begin{equation}
\inf (P_1)\geq\inf (CP_1)\geq \sup (CP_1^*) = \sup (D_1).\end{equation}
Therefore, it is easy to see that zero duality gap between $(P_1)$ and the semi-Lagrangian dual problem $(D_1)$ implies that
\begin{equation} \inf (P_1)= \inf (CP_1)= \sup (CP_1^*) = \sup (D_1),\end{equation}
and in particular, exact copositive relaxation and exact completely positive relaxation.

\section{Conclusion and Future Work}
In this paper, we have identified that convexifiability of nonconvex QPs forms the basis for zero duality gaps between nonconvex QPs and their semi-Lagrangian duals and have provided classes of nonconvex quadratic optimization problems, admitting convexifiability and consequently zero duality gaps under suitable conditions.
In particular, we have established that convexifiablity is hidden in some mixed integer quadratic programs and robust mixed integer quadratic optimization problems, guaranteeing zero duality gap.

Our approach and results highlight the significance of convexifiability of nonconvex quadratic programs that allows identification of classes of discrete, robust and continuous quadratic programs with non-negative variables, exhibiting gap-free duals under suitable conditions. It also shows promise of extensions of zero duality gap results to quadratic optimization problems with conic constraints and to polynomial optimization problems.

An interesting future research direction is to examine zero duality gap for mixed-integer quadratic optimization problems in the face of more general uncertainty sets, such as ellipsoidal data uncertainty,  and for multi-stage robust optimization problems \cite{Tutorial} which are increasingly becoming common in modelling real-world decision-making problems of optimization in the face of data uncertainty.  These will be investigated in a forthcoming study.

\end{document}